\documentclass[12pt]{amsart}

\pdfoutput=1

\usepackage[text={420pt,660pt},centering]{geometry}

\usepackage{color}
\usepackage{stmaryrd}
\usepackage[mathscr]{euscript}
\usepackage{esint,amssymb}
\usepackage{graphicx}
\usepackage{MnSymbol}
\usepackage{mathtools}
\usepackage[colorlinks=true, pdfstartview=FitV, linkcolor=blue, citecolor=blue, urlcolor=blue,pagebackref=false]{hyperref}
\usepackage{microtype}

\usepackage{bm}
\usepackage{scalerel} %for scaling the boxes 
\usepackage{dsfont}
\usepackage[font={footnotesize}]{caption}

\definecolor{darkgreen}{rgb}{0,0.5,0}
\definecolor{darkblue}{rgb}{0,0,0.7}
\definecolor{darkred}{rgb}{0.9,0.1,0.1}

\newtheorem{proposition}{Proposition}
\newtheorem{theorem}[proposition]{Theorem}
\newtheorem{lemma}[proposition]{Lemma}
\newtheorem{corollary}[proposition]{Corollary}

\theoremstyle{definition}
\newtheorem{remark}[proposition]{Remark}

\newcommand{\cref}[1]{Corollary~\ref{c.#1}}

\numberwithin{equation}{section}
\numberwithin{proposition}{section}
\newcommand{\N}{\mathbb{N}}
\newcommand{\R}{\mathbb{R}}

\newcommand{\E}{\mathbb{E}}
\renewcommand{\P}{\mathbb{P}}

\newcommand{\Zd}{\mathbb{Z}^d}
\newcommand{\Rd}{{\mathbb{R}^d}}

\newcommand{\eps}{\varepsilon}

\renewcommand{\a}{\mathbf{a}}

\newcommand{\g}{\mathbf{g}}
\newcommand{\f}{\mathbf{f}}

\renewcommand{\subset}{\subseteq}

\renewcommand{\a}{\mathbf{a}}

\newcommand{\ahom}{{\overline{\mathbf{a}}}}

\newcommand{\chom}{{\overline{c}}}

\renewcommand{\subset}{\subseteq}

\newcommand{\cu}{{\scaleobj{1.2	}{\square}}}

\renewcommand{\L}{\mathcal{L}}
\renewcommand{\fint}{\strokedint}

\DeclareMathOperator{\var}{var}

\renewcommand{\bar}{\overline}

\renewcommand{\tilde}{\widetilde}

\newcommand{\indc}{\mathds{1}}

\renewcommand{\hat}{\widehat}

\begin{document}
\title{Local central limit theorem for gradient field models}

\begin{abstract}
We consider the gradient field model in $\left[ -N,N\right]
^{2}\cap \mathbb{Z}^{2}$ with a uniformly convex interaction potential.   Naddaf-Spencer \cite{NS} and Miller \cite{Mi} proved that the macroscopic averages of linear statistics of the field converge to a continuum Gaussian free field. In this paper we prove the distribution of $\phi(0)/\sqrt{\log N}$ converges uniformly in $\R$ to a Gaussian density,  with a Berry-Esseen type bound. This implies the distribution of $\phi(0)$ is sufficiently `Gaussian like' between $[-\sqrt {\log N}, \sqrt {\log N}]$. 
\end{abstract}

\author[W. Wu]{Wei Wu}
\address[W. Wu]{Mathematics Department, NYU Shanghai \& NYU-ECNU Institute of Mathematical Sciences,  China}
\email{wei.wu@nyu.edu}
\maketitle

\section{Introduction}

In this paper we study  a two dimensional
 gradient interface field with a nearest neighbor potential $V$.  Explicitly, let $Q_{N}:=\left[ -N,N\right]
^{2}\cap \mathbb{Z}^{2}$ and let the boundary $\partial Q_{N}$ consist of
the vertices in $Q_{N}$ that are connected to $\mathbb{Z}^{2}\setminus Q_{N}$
by an edge. The gradient field on $Q_{N}$ with zero boundary
condition is a random field denoted by $\phi ^{Q_{N},0}$, whose distribution
is given by the Gibbs measure 
\begin{equation}
d\mu _{N}=Z_{N}^{-1}\exp \left[ -\sum_{v\in Q_{N}}\sum_{i=1}^{2}V\left(
\nabla _{i}\phi \left( v\right) \right) \right] \prod_{v\in Q_{N}\backslash
\partial Q_{N}}d\phi \left( v\right) \prod_{v\in \partial Q_{N}}\delta
_{0}\left( \phi \left( v\right) \right) ,  \label{e.GL}
\end{equation}%
where $\nabla _{i}\phi \left( v\right) =\phi \left( v+e_{i}\right) -\phi
\left( v\right) $, {$e_{1}=(1,0)$ and $e_{2}=(0,1)$}, and we set $\phi
\left( v\right) =0$ for all $v\in \mathbb{Z}^{2}\setminus Q_{N}$. Here $%
Z_{N} $ is the normalizing constant ensuring that $\mu _{N}$ is a
probability measure, i.e. $\mu _{N}(\mathbb{R}^{|Q_{N}|})=1$. We denote
expectation and variance with respect to $\mu _{N}$ by $\langle \cdot \rangle_{\mu_N}$ and $\var_{\mu_N}$,
respectively.

We assume the interaction potential $V:\R\to\R$ in \eqref{e.GL} satisfies the following:
\begin{enumerate}

\item[(i)]\emph{Symmetry}: for every $t\in\R$, we have $V(t) = V(-t)$.

\smallskip

\item[(ii)] \emph{Uniform convexity}: for every $t\in\R$, we have $0< \lambda \leq V''(t) \leq \Lambda<\infty$.

\smallskip

\item[(iii)] \emph{Regularity}: $V\in C^{2,1}(\R)$.  In other words,  $\mathsf V''$ is Lipschitz continuous with Lipschitz constant $L$.   
\end{enumerate}

\smallskip
The Gibbs measure \eqref{e.GL} was introduced in the 1970s by Brascamp, Lieb and Lebowitz \cite{BLL},  in the name of anharmonic crystals.  Since then, numerous efforts have been made to study the large-scale (macroscopic) statistical behavior of the field $\nabla \phi$.  Notable progress was made by Naddaf and Spencer \cite{NS},  who studied the infinite volume limit of the Gibbs measure \eqref{e.GL} (the infinite volume Gibbs states were rigorously characterized by Funaki and Spohn \cite{FS}),  and proved a central limit for (rescaled) linear functions of $\nabla \phi$.  More precisely,  they consider, for $R\geq 1$, the random variable
\begin{equation}
F_R(\nabla \phi) := R^{-\frac d2}\sum_{x\in\Zd} \sum_{i=1}^d f_i\left(\frac xR\right) \left( \phi(x+e_i)-\phi(x) \right), 
\end{equation}
where for each $i\in\{1,\ldots,d\}$,  take $f_i: \Rd \to \R$ to be a compactly supported, smooth, deterministic function,  and proved that the random variable $ F_R$ converges in law to a normal random variable.  Later,  the result of \cite{NS} has been generalized to dynamical settings \cite{GOS} and also to finite volume measures such that the boundary condition has at most poly-logarithmic fluctuations \cite{Mi}.  We also mention the related work \cite{DGI} which characterizes the Wulff shape and the large deviation principle for macroscopic profiles of $\nabla \phi$,  and results that are related to the extension of the Gibbs measure \eqref{e.GL} to non-uniformly convex settings \cite{She, BS, AKMu, MiP, MaP, Bau1,Bau2},   the study of the maximum of the random field \eqref{e.GL} \cite{BW,WZ}, and to spin models which is related to a Gibbs measure with convex interaction by a duality transform \cite{DW}. 

A natural question arises is whether the Gaussian limit established in \cite{NS} holds on much smaller scales, given that the microscopic interaction is not Gaussian.  Under an additional assumption on the ellipticity contrast for $V$,  namely $\Lambda < 2\lambda$,  Conlon and Spencer proved a stronger result \cite{CS},  which states that for the infinite volume gradient Gibbs measure $\mu$,
\begin{equation*}
\left|  \log \left\langle \exp(t(\phi(0) - \phi(x))) \right\rangle_\mu 
- \frac {t^2}2 \var_\mu(\phi(0) - \phi(x))
\right|
\leq 
Ct^3 \sup_{x\in\R} |V'''(x)|.
\end{equation*}
The result suggests that under these assumptions of $V$,  the pointwise distribution of $\phi(x) -\phi(0)$ is close to a Gaussian,  and one has to move to a large deviation regime to see non-Gaussian tails.

In this paper we are able to bring down the scale to $R=1$,  and prove a local central limit theorem of the gradient Gibbs measure \eqref{e.GL},   under the assumptions on the potential~$V$ given at the beginning of the paper. 
\begin{theorem}
\label{t.main}
Let $\phi $ be sampled from the Gibbs measure (\ref{e.GL}). Assume
the potential $\mathsf V\left( \cdot \right) $ satisfies the coniditions (i) -(iii). 
Then the density function $g_N$ of $\phi(0)/ \sqrt{\log N}$ converges uniformly in $\R$ to $\frac{1}{\sqrt{2\pi}} e^{-\frac1{2 \g} x^2}$.  Moreover, there exists $C<\infty$ such that 
$|g_N(x) - \frac{1}{\sqrt{2\pi}} e^{-\frac1{2\g} x^2} | \leq   \frac {C}{\sqrt{\log N}}$
\end{theorem}

We remark here the same proof (with a slightly modification of the multiscale argument in Section \ref{s.decouple} ) also gives the local CLT at other locations inside the bulk.  Namely,  for any $x_N\in Q_N$ such that the graph distance between $x_N$ and $\partial Q_N$ tends to infinity,  the density of $\frac{\phi(x_N)}{\sqrt{\log \text{dist}(x_N,  \partial Q_N)}}$ converges uniformly in $\R$ to the same Gaussian limit.   The same proof also works in the inifnite volume setting,  gives the local CLT for $\frac{\phi(0)- \phi(x)}{(\log |x|)^\frac 12}$ as $|x|\to \infty$.  

Notice that if $\phi(0)$ can be written as $X_1 + \cdots X_{\log N}$,  where $X_i$ are i.i.d random variables, then the Berry-Esseen Theorem gives $|g_N(x) - \frac{1}{\sqrt{2\pi}} e^{-\frac1{2\g} x^2} | \leq   \frac {C\E[|X_1|^3]}{\sqrt{\log N}}$.  We will see in Section \ref{s.decouple} that the analogues of $X_i$ are increments of harmonic averages of $\phi$,    which are not independent but have certain decoupling properties (thanks to \cite{Mi}),  and Theorem \ref{t.main} gives a Berry-Esseen type estimate for the density function.  An immediate consequence of Theorem \ref{t.main} is that the distribution of $\phi(0)$ is sufficiently spreadout in $[-\sqrt{\log N}, \sqrt{\log N}]$. Indeed,
given any $a\in [-\sqrt{\log N}, \sqrt{\log N}]$, we apply Theorem \ref{t.main} to obtain

\begin{equation*}
\int_{\frac{a}{\sqrt{\log N}}}^{\frac{a+1}{\sqrt{\log N}}} g_N(x) \,dx - \int_{\frac{a}{\sqrt{\log N}}}^{\frac{a+1}{\sqrt{\log N}}}\frac{1}{\sqrt{2\pi}} e^{-\frac1{2 \g} x^2} \,dx 
= O(\frac{1}{\log N})
\end{equation*}
Thus %there exist constants $C<\infty$,  independent of $a, N$,  such that 
\begin{equation*}
 \P(\phi(0) \in [a,a+1]) = \frac{1}{\sqrt{2\pi \log N}}e^{-\frac1{2 \g} \frac{a^2}{\log N} } + O(\frac{1}{\log N})
\end{equation*}

Why do we expect a local CLT at scale $O(1)$? The argument to prove a macroscopic CLT in \cite{NS} was based on a beautiful observation that the scaling limit can be derived from an elliptic homogenization problem via the Helffer-Sj\"ostrand representation \cite{HS}.  With Armstrong \cite{AW},  we extend and quantify the homogenization argument by Naddaf and Spencer, based on the quantitative theory for homogenization developed by Armstrong, Kuusi and Mourrat \cite{AKM, AKMBook}.  In particular,  we obtained the convergence of the Hessian of the surface tension with an algebraic rate,  resolve an open question posed by Funaki and Spohn \cite{FS} regarding the $C^2$ regularity of surface tension,  and the fluctuation-dissipation conjecture of \cite{GOS}.  Following the approach of \cite{AW} and \cite{AKMBook},  we are able to obtain a quantitative homogenization of the Helffer-Sj\"ostrand PDE,  thus estimate the covariance structure of $\mu_N$, as $N$ gets large,  with a high precision (we also refer to \cite{AD2} for some further applications of the quantitative homgenization ideas to the $\nabla\phi$-model).   To obtain a Gaussian limit of $\phi(0)/ \sqrt{\log N}$, 
we apply the harmonic approximation result by Miller \cite{Mi}, which enables us to write $\phi(0)$ as the sum of $\log N$ increments,  with certain decoupling properties.  These are the two main ingredients behind the proof of Theorem \ref{t.main}.

We remark here that the key estimate for proving Theorem \ref{t.main} is the characteristic function asymptotics for $\left\langle \exp\left( i\frac{t}{\sqrt{\log N}} \phi(0) \right) 
\right\rangle_{\mu_N}$, given in Lemma \ref{l.small}.  If one has the convergence stated in Lemma  \ref{l.small} for $t\in\R$, without quantifying the rate of convergence (namely,  a CLT for $\phi(0)/ \sqrt{\log N}$), then it implies a local CLT without any rate of convergence.  The rate of convergence for the local CLT depdends on the convergence rate and the validity of the $t$ region such that Lemma  \ref{l.small} holds.  In this paper we present a proof with a convergence rate using quantitative homogenization,  and explain briefly in the next section how it can be simplified if one only aims for a qualitative local CLT.

The proof of Theorem \ref{t.main} presented in this paper implies the asymptotics for the characteristic function for the distribution of $\phi(0)$,  namely there exists some $\g= \g(V)>0$, such that  
\begin{equation}
\label{e.char}
\left\langle \exp\left( is \phi(0) \right) 
\right\rangle_{\mu_N}
\approx
e^{- \frac{s^2}{2} \g \log N},
\end{equation}
as long as $s = O((\log N)^{-\frac 14})$.  Using the same major ingredients,  but with a more elaborate multiscale argument,  we may improve the \eqref{e.char} so that it holds for $s$ within a small neighborhood of the origin,  with radius independent of $N$.  It would be very interesting to extend the characteristic function estimates \eqref{e.char} beyond $|s| =1$.  For this,  we include an open question posed by Tom Spencer. 
\newline

\textbf{Open Question:}  Consider the lattice dipole gas model,  which is a special case of the gradient Gibbs measure \eqref{e.GL} with $V(x) = \frac{x^2}{2} + z \cos x$,  with $|z|<1$.  Prove that the leading term in the asymptotic expansion of the characteristic function 
\begin{equation}
\label{e.dipole}
\left\langle \exp\left( i\pi ( \phi(0) -\phi(x)) \right) 
\right\rangle_{\mu_{dipole}}
\approx
e^{- \frac{\pi^2}{2} \g \log |x|},
\end{equation}
for some $\g >0$.  
Estimates of type \eqref{e.dipole} plays an important role in the study of the (lattice) Coulomb gas and of the Coulomb gas representation of the low temperature Abelian spin models. 

We summarize the characteristic function estimates needed for proving Theorem \ref{t.main} in the next section,  and introduce some notations in Section \ref{s.prelim}.  In Section \ref{s.CLT} we derive a central limit theorem for the linear statistics of $\nabla\phi$ with an algebraic rate of convergence, which quantifies the result of \cite{NS}, following the argument of \cite{AW} and \cite{AKMBook}.  In Section \ref{s.decouple} we recall the harmonic approximation result in \cite{Mi},  and use a multiscale argument to obtain the precise characteristic function asymptotic of $\phi(0)$.  Finally,  we gave an upper bound for the large $s$ characteristic function in Section \ref{s.MW},  based on a Mermin-Wagner type argument,  and finishes the proof of Theorem \ref{t.main} . 

\section{Estimates for characteristic functions}

Theorem \ref{t.main} follows from the quantitative estimates of the characteristic function below. 
The first lemma gives the precise estimate for the characteristic function $\langle e^{is\phi(0)} \rangle_{\mu_N}$ for $s = o((\log N)^{-\frac 14})$.  
\begin{lemma}
\label{l.small}
There exists $\g=\g(V) >0$,  such that for $N$ sufficiently large and $t= o((\log N)^{\frac 14})$, we have 
\begin{equation}
\left\langle \exp\left( i\frac{t}{\sqrt{\log N}} \phi(0) \right) 
\right\rangle_{\mu_N}
= e^{- \frac{t^2}{2} \g} \left( 1+O\left(\frac{t^2}{(\log N)^\frac 12}\right)\right)
\end{equation}
\end{lemma}

\begin{remark}
Lemma \ref{l.small} quantifies the CLT for the pointwise field $\phi(0)/ \sqrt{\log N}$, namely for $t\in\R$, 
\begin{equation}
\label{e.clt}
\left\langle \exp\left( i\frac{t}{\sqrt{\log N}} \phi(0) \right) 
\right\rangle_{\mu_N}
= e^{- \frac{t^2}{2} \g} +o_N(1).
\end{equation}
The pointwise CLT \eqref{e.clt} can be proved by combining the CLT for the macroscopic average of the field established in \cite{NS, Mi}, and the multscale argument presented in \cite{BW} or in Section \ref{s.decouple} of this paper. 
\end{remark} 

We also need (non-optimal) decay estimate of the characteristic function for large $s$,  summarized in the two lemmas below.  
\begin{lemma}
\label{l.medium}
Let $\g=\g(V) >0$ be the same constant as in Lemma \ref{l.small}. There exists $\eps = \eps(V)>0$ , such that for $|s|<\eps$, we have for $N$ sufficiently large, 
\begin{equation}
\left\langle \exp\left( is \phi(0) \right) 
\right\rangle_{\mu_N}
\leq 2e^{- \frac{s^2}{3} \g \log N}
\end{equation}
\end{lemma}

\begin{remark}
If one only aims for the qualitative local CLT,  then it suffices to prove a weaker estimate,  that there exist $c_1>0$ and  $\eps = \eps(V)>0$ , such that for $|s|<\eps$, we have for $N$ sufficiently large, 
\begin{equation}
\label{e.mediumweak}
\left\langle \exp\left( is \phi(0) \right) 
\right\rangle_{\mu_N}
\leq 2e^{- c_1 s^2  \log N}.
\end{equation}
As will be explained in Section \ref{s.decouple},  the proof of \eqref{e.mediumweak} is simpler,  and the quantitative CLT presented in Section \ref{s.CLT} would not be needed.  
\end{remark}

\begin{lemma}
\label{l.MW}
There exists $\eps_1 >0$ and $C<\infty$, such that for $s\in\R$, we have 
\begin{equation}
\label{e. MW}
|\left\langle \exp\left( is \phi(0) \right) \right\rangle_{\mu_N}|
\le
\min\{1-\eps_1, \frac C{s^2}\}^{\log N}
\end{equation}
\end{lemma}

Before proving these lemmas,  we now explain how they imply Theorem \ref{t.main}.  

\begin{proof}
[Proof of Theorem \ref{t.main} with no rate]
To prove Theorem \ref{t.main},  arguing as the classical local CLT and write 
\begin{equation}
\label{e.Psi}
\Psi_N (t) := \left\langle \exp\left( it \phi(0)/\sqrt{\log N} \right) \right\rangle_{\mu_N}
\end{equation}
Then by inversion theorm,
\begin{equation*}
|g_N(x) - \frac{1}{\sqrt{2\pi}} e^{-\frac1{2\g} x^2} |
=
\left| \int_{-\infty} ^\infty  \frac 1{2\pi} \left( \Psi_N(t) e^{itx} - e^{-\frac 12 \g t^2} e^{itx} \right) \,dt\right|
 \leq
\frac 1{2\pi} \int_{-\infty}^\infty \left| \Psi_N (t) - e^{-\frac 12 \g t^2} \right| \,dt
\end{equation*}
 
 We claim the right side above goes to zero by split the integral into three parts: 
 
 - For $|t|\le a $,  we apply \eqref{e.clt} to conclude that 
 $\Psi_N(t)$ goes to $ e^{-\frac 12 \g t^2} $ uniformly for $t\in [-a,a]$. Thus
 \begin{equation}
 \label{e.bounded}
 \int_{-a}^a \left| \Psi_N (t) - e^{-\frac 12 \g t^2} \right| \,dt \to 0 
\end{equation}
as $N\to\infty$. 

- For $a \le |t| \le \eps \sqrt{\log N}$, we apply \eqref{e.mediumweak},  which yields
 \begin{equation}
\Psi_N (t)  \leq C  e^{-c_1 t^2}
\end{equation}
 Thus 
  \begin{equation}
 \int_{a \le |t| \le \eps \sqrt{\log N}} \left| \Psi_N (t) - e^{-\frac 12 \g t^2} \right| \,dt 
 \leq 2C \int_a^\infty e^{-c_1 t^2} \, dt \to 0 
\end{equation}
if we take $a\to \infty$.

- For $|t| \ge \eps \sqrt{\log N}$, we use the bound \eqref{e.  MW} which implies (let $s = t/\sqrt{\log N}$)
 \begin{multline*}
 \int_{|t| > \eps \sqrt{\log N}} \left| \Psi_N (t) \right| \, dt
 =    \sqrt{\log N} \int_{|s|>\eps} \left| \Psi_N (s) \right| \, ds \\
 \leq
  \sqrt{\log N}  \left( \int_{C>|s|>\eps} (1-\eps_1)^{\log N} \, ds
  + \int_{|s|>C} (\frac C{s^2})^{\log N} \, ds \right)
 \end{multline*}
 which goes to $0$ as $N \to \infty$.
 And we conclude Theorem \ref{t.main}.
 \end{proof}
 
 \begin{proof}[Quantitative proof of Theorem \ref{t.main}]
 To quantify the rate of convergence for the local CLT,  take $a_N= \sqrt{\frac 3{2\g}\log\log N} $ in the proof above.  
  
   - For $|t|\le a_N $,  we apply Lemma \ref{l.small} to obtain a rate of convergence that 
  \begin{equation}
  \int_{-a_N}^{a_N} \left| \Psi_N (t) - e^{-\frac 12 \g t^2} \right| \,dt
  \leq
  \int_{-a_N}^{a_N} C e^{- \frac{t^2}{2} \g} \frac{t^2}{(\log N)^\frac 12}\,dt
  \leq
   \frac{C}{(\log N)^\frac 12}
  \end{equation}
for $N$ sufficiently large. 
  
  - For $a_N \le |t| \le \eps \sqrt{\log N}$, we apply Lemma \ref{l.medium},  which yields
 \begin{equation}
\Psi_N (t)  \leq 2  e^{-\frac13 \g t^2}
\end{equation}
 Thus 
  \begin{equation}
 \int_{a_N \le |t| \le \eps \sqrt{\log N}} \left| \Psi_N (t) - e^{-\frac 12 \g t^2} \right| \,dt 
 \leq 3 \int_{a_N}^\infty e^{-\frac13 \g t^2} \, dt
 \leq
 \frac{C}{(\log N)^\frac 12}
\end{equation}
for $N$ sufficiently large. 

 Combine with the estimates for $|t| \ge \eps \sqrt{\log N}$ in the qualitative proof above,  we conclude $|g_N(x) - \frac{1}{\sqrt{2\pi}} e^{-\frac1{2\g} x^2} | =O((\log N)^{-\frac 12})$.
  \end{proof}

 Lemma \ref{l.small} and \ref{l.medium} will be proved in Section \ref{s.smallmedium},  whereas Lemma 
 \ref{l.MW} will be proved in Section \ref{s.MW}.

 \section{Preliminaries and Notation}
 \label{s.prelim}
 Given a set $U \subset Q_R$,  we let~$\mathcal{E}(U)$ denote the set of directed edges on~$U$ and $U^\circ$  the \emph{interior} of $U$.  Define $\Omega_0(U)$ to be the set of functions $\phi:U \to \R$ such that $\phi=0$ on $\partial U$.  Given $e=(x,y)\in \mathcal{E}(U)$ and $\phi\in \R^U$, we define $\nabla \phi(e):= \phi(y) - \phi(x)$.  The formal adjoint~$\nabla^*$ of~$\nabla$, which is the discrete version of the negative of the divergence operator, is defined for functions~$\g:\mathcal{E}(U)\to \R$ by 
\begin{equation}
\left( \nabla^*\g\right)(x)
:=
\sum_{e\ni x} \g(e) , 
\quad x\in U^\circ.
\end{equation}

The average of a function $f: U\to\R$ on $U$ is denoted as $(f)_U := \frac 1{|U|} \sum_{x\in U} f(x) $. 

We define, for each $x\in U$, the basis element $\omega_x\in \Omega_0(U)$ by 
\begin{equation*}
\omega_x(y):= \left\{ 
\begin{aligned}
& 1 & \mbox{if} \ x=y,\\
& 0 & \mbox{if} \ x\neq y,
\end{aligned}
\right. 
\end{equation*}
and the differential operator~$\partial_x$ by 
\begin{equation}
\label{e.partial}
\partial_x u (\phi):= \lim_{h\to 0} \frac1h \left( u(\phi+h\omega_x) - u(\phi) \right).
\end{equation} 

Define $L^p(\mu)$ to be the set of measurable functions $u:\Omega_0(U) \to \R$ such that 
\begin{equation*}
\left\| u \right\|_{L^p(\mu)} 
: = \left( \int_\Omega \left| u(\phi)\right|^p \,d\mu(\phi) \right)^{\frac1p}
< +\infty.
\end{equation*}
We define $H^1(\mu)$ to be 
\begin{equation*}
\left\| u \right\|_{H^1(\mu)}
:=
\left(
\left\| u \right\|_{L^2(\mu)}^2 
+
\sum_{x\in U^\circ} \left\| \partial_x u \right\|_{L^2(\mu)}^2
\right)^{\frac12}.
\end{equation*}
We let $H^{-1}(\mu)$ denote the dual space of~$H^1(\mu)$, that is, the closure of~$C^\infty(\Omega_0(U))$ functions under the norm
\begin{equation*}
\left\| 
w
\right\|_{H^{-1}(\mu)}:= 
\sup
\left\{ 
\int_\Omega u(\phi) w(\phi) \,d\mu(\phi)
\,:\,
u\in H^1(\mu), \ \left\| u \right\|_{H^1(\mu)} \leq 1
\right\}. 
\end{equation*}
We define the space $L^2(U,\mu)= L^2(U;L^2(\mu))$ to be the set of measurable functions $u:U\times \Omega_0(U) \to \R$ with respect to the norm 
\begin{equation*}
\left\| u \right\|_{L^2(U,\mu)}
:=
\left(
\sum_{x\in U} \left\| u(x,\cdot) \right\|_{L^2(\mu)}^2 \right)^{\frac12}.
\end{equation*}
We also define $H^1(U,\mu)$ by the norm
\begin{equation*}
\left\| u \right\|_{H^1(U,\mu)}
:=
\left( 
\sum_{x\in U} 
\left\| u(x,\cdot) \right\|_{H^1(\mu)}^2
+
\sum_{e\in \mathcal{E}(U)} 
\left\| \nabla u(e,\cdot) \right\|_{L^2(\mu)}^2 
\right)^{\frac12} 
\end{equation*}
The subset $H^1_0(U,\mu) \subseteq H^1(U,\mu)$ consists of those functions $u\in H^1(U,\mu)$ which satisfy
$u(x,\phi) = 0$ for every $\partial U\times \Omega_0(U)$. 

We define $H^{-1}(U,\mu)$ to be the dual space of $H^1_0(U,\mu)$. 
That is, $H^{-1}(U,\mu)$ is the closure of smooth functions
with respect to the norm
\begin{equation*}
\left\| 
w
\right\|_{H^{-1}(U,\mu)}:= 
\sup
\left\{ 
\sum_{x\in U} \int_{\Omega_0(U)} u(x,\phi) w(x,\phi) \,d\mu(\phi)
\,:\,
u\in H^1_0(U,\mu), \ \left\| u \right\|_{H^1(U,\mu)} \leq 1
\right\}. 
\end{equation*}
It is sometimes convenient to work with the volume-normalized versions of the $L^2$ and Sobolev norms, defined by 
\begin{equation*}
\left\| u \right\|_{\underline{L}^2(U,\mu)}
:=
\left(
\frac1{|U|}
\sum_{x\in U} \left\| u(x,\cdot) \right\|_{L^2(\mu)}^2 \right)^{\frac12},
\end{equation*}
\begin{equation*}
\left\| u \right\|_{\underline{H}^1(U,\mu)}
:=
\left( 
\frac1{|U|}
\sum_{x\in U} 
\left\| u(x,\cdot) \right\|_{H^1(\mu)}^2
+
\frac1{|U|}\sum_{e\in\mathcal{E}(U)} 
\left\| \nabla u(e,\cdot) \right\|_{L^2(\mu)}^2 
\right)^{\frac12},
\end{equation*}
\begin{multline*}
\left\| 
w
\right\|_{\underline{H}^{-1}(U,\mu)}
\\
:= 
\sup
\left\{ 
\frac1{|U|}\sum_{x\in U} \int_{\Omega_0(U)} u(x,\phi) w(x,\phi) \,d\mu(\phi)
\,:\,
u\in H^1_0(U,\mu), \ \left\| u \right\|_{\underline{H}^1(U,\mu)} \leq 1
\right\}. 
\end{multline*}

We notice that the formal adjoint of~$\partial_x$ with respect to $\mu_{N}$, which we denote as $\partial_x^*$, is given by 
\begin{equation*}
\partial_x^* w := -\partial_x w 
+
\sum_{y\sim x} 
V'(\phi(y)-\phi(x) - \xi \cdot (y-x)) w(\phi).
\end{equation*}
This can be easily checked by the identity for all $u,v \in H^1(\mu_{N})$ that
\begin{equation*}
\left\langle (\partial_x u) v \right\rangle_{\mu_{N}}
= \left\langle u (\partial^*_x v)  \right\rangle_{\mu_{N}}.
\end{equation*}
We also have the commutator identity
\begin{equation}
\label{e.commutator}
\left[ \partial_x, \partial_y^* \right]  
=
- \indc_{\{x \sim y\}}V''\big( \phi(y) - \phi(x) - \xi \cdot (y-x) \big)
+ \indc_{\{x=y\}} \sum_{e \ni x} V''\left( \nabla \phi(e) - \nabla \ell_\xi(e)\right).
\end{equation}
Define the Witten Laplacian $\L_{\mu_N}$ as
\begin{equation*}
\L_{\mu_N} F = - \sum_{x\in Q_N^\circ} \partial_x^* \partial_x F, 
\end{equation*}

For every cube $Q \subseteq\Zd$ and $u,v\in H^1(Q,\mu)$, we define 
\begin{equation}
\label{e.defBU}
\mathsf{B}_{\mu, Q}\left[u,v\right]
:= 
\frac{1}{|Q|}
\sum_{y\in Q_N^\circ} 
\sum_{x\in Q^\circ}
\left( \partial_y u(x,\cdot), \partial_y v(x,\cdot)\right)_{\mu} 
+ \frac1{|Q|}\sum_{e\in \mathcal{E}(Q)} \left\langle 
 \nabla u(e,\cdot) V''(e,\cdot) \nabla v(e,\cdot) 
\right\rangle_{\mu}
\end{equation}
and
\begin{align*}
\mathsf{E}_{\mu, Q,\f} \left[u\right]
& 
:=
\frac12 \mathsf{B}_{\mu, Q}\left[u,u\right]
- \frac1{|Q|}\sum_{e\in \mathcal{E}(Q)}
\left\langle \f(e,\phi) \nabla u(e,\cdot)  \right\rangle_{\mu} .
\end{align*}
 
 For $D\subset Q_R$,  and $f: \partial D \to \R$, define the
$\nabla \phi$ measure on $D$ with Dirichlet boundary condition $f$ by 
\begin{equation}
\label{e.GLD}
d\mu _{D}^{f}=Z_{D}^{-1}\exp \left[ -\sum_{x\in D}\sum_{i=1}^{2}V\left(
\nabla _{i}\phi \left( x\right) \right) \right] \prod_{x\in D\backslash
\partial D}d\phi \left( x\right) \prod_{x\in \partial D}\delta _{0}\left(
\phi \left( x\right) -f\left( x\right) \right) .  
\end{equation}

Here $%
Z_{D} $ is the normalizing constant ensuring that $\mu _{D}^f$ is a
probability measure. We denote
expectation and variance with respect to $\mu _{D}^f$ by $\E^{D,f}$ and $\var_{D,f}$,
respectively.
 
 We finally present the Brascamp-Lieb inequality \cite{BL,NS}, which states that the variance of observables with respect to a log-concave measure is dominated by that of a Gaussian measure. We denote the Green function for the discrete Laplacian with zero Dirichlet boundary conditions in~$Q_L$ by $G_{Q_L}(x,y)$.

\begin{proposition}[Brascamp-Lieb inequality for $\mu_L$]
\label{p.BL} 
For every $F\in H^1(\mu_L)$,
\begin{equation}
\label{e.BL.var}
\var_{\mu_L} \left[ F \right] 
\leq 
\frac1\lambda 
\sum_{x,y \in Q_L^\circ}G_{Q_L}(x,y) \left\langle \left( \partial
_{x}F\right) \left( \partial _{y}F\right) \right\rangle_{\mu_L}.
\end{equation}
\item
For every $f \in \R^{Q_L}$, we have 
\begin{equation}
\label{e.BL.linexp}
\log \left \langle  \exp (t \sum_{y\in Q_L} \phi(y)f(y) ) \right \rangle_{\mu_L} 
\leq 
\frac{t^{2}}{2\lambda} \sum_{x,y \in Q_L^\circ}
G_{Q_L}(x,y) f(x)f(y)
\end{equation}
\end{proposition}

We sometimes denote by $\mu_{G,D}^f$ the finite volume Gaussian measure in $D$ (i.e.,  the special case of \eqref{e.GLD} with $V(x) = \frac 12 x^2$).  We denote the corresponding expectation 
and variance by $\E^{G,D,f}$ and $\var_{G,D,f}$ respectively.  When $f=0$ we will omit its appearence on the supercripts. 
 
 \section{Quantitative convergence of the variance}
 \label{s.CLT}
 A main ingredient for the refined estimate of $\Psi_N$, defined in \eqref{e.Psi} is the following convergence of the variance of the linear statisitcs of $\nabla \phi$, with an algebraic rate. 
 
  \begin{theorem}
[{Quantitative convergence of variance}]
\label{t.qclt}
Fix $R \in [1,\infty)$.  Let $\phi$ be sampled from the finite volume Gibbs measure $\mu_R$ with zero boundary condition \eqref{e.GL}.
Let $f_R: Q_R \to \R$ and  $f\in L^2([0,1]^2)$  be such that there exists $\alpha>0$,  so that  $\|\nabla^*\cdot f_R(\frac \cdot R) - f(\cdot)\|_{L^\infty([0,1]^2)} \le R^{-\alpha}$.
Define the random variable 
\begin{equation*}
\Phi_R \left(f\right) 
:=
R^{-d/2}\sum_{e\in \mathcal E( Q_R)} \nabla \phi(e) \cdot f_R\left( e\right).
\end{equation*}
Then there exists $\g= \g_{V,f}>0$, 
 $\beta =\beta(d, \alpha, \lambda, \Lambda) \in \left(0,\tfrac12\right]$ and $C_0(d, \lambda, \Lambda) < \infty$  such that, for every $R\in [1,\infty)$, 
\begin{equation*}
\left| \var_{\mu_R} [\Phi_R \left(f \right) ]  - \g \right|
\leq 
C_0 R^{-\beta } \|\nabla^*\cdot f_R\|_{L^\infty} . 
\end{equation*}

  \end{theorem}
  
  \begin{remark}
  The central limit theorem for the $\nabla \phi$ model, i.e., the convergence of $\Phi_R$ in distribution to a normal random variable,  was established in \cite{NS, GOS, Mi},  without quantifying the rate of convergence. \end{remark}
  \begin{remark}
  It will be clear from Theorem \ref{t.HS} below that $\mathsf g$ can be explicity written as 
 \begin{equation*}
 \mathsf g = \int_{[0,1]^2 \times [0,1]^2} f(x) (\nabla^*\cdot \ahom \nabla )^{-1} (x,y) f(y) \, dxdy,
 \end{equation*}
 for some positive definite matrix $\ahom = \ahom (V)$. 
  \end{remark}

Theorem \ref{t.qclt} follows from homogenization of an elliptic PDE based on the convergence results of  \cite{AW},  as we explain below.  The starting observation is the variational characterization of the variance,  known as the Helffer-Sj\"ostrand representation (see \cite{NS,  AW}), which gives 
\begin{equation}
\label{e.variational}
\var_{\mu_R} [R^{-d/2}\sum_{e\in \mathcal E( Q_R)} \nabla \phi(e) \cdot f_R\left( e\right) ]   
=
-2 \inf_{w \in H^1_0(Q_R,\mu_{R})} \mathsf{E}_{\mu_{R},Q_R,f_R} \left[w\right],
\end{equation}
where we let $\mathsf{E}_{\mu ,U,f} \left[\cdot \right]$ denote the energy functional
\begin{align*}
\mathsf{E}_{\mu,U,f} \left[w\right]
& 
:=
\frac12 \sum_{y\in Q} \sum_{x\in U^\circ} \left\langle (\partial_y w(x,\cdot) )^2 \right\rangle_{\mu }
+
\frac12 \sum_{e\in \mathcal{E}(U)} 
\left\langle V''(e) (\nabla w(e,\cdot))^2  \right\rangle_{\mu } 
\\ & \qquad 
- \sum_{x\in U^\circ} \left\langle f(x,\cdot) w(x,\cdot) \right\rangle_{\mu }.
\end{align*}
The minimizer of \eqref{e.variational} can be written as $R^{-\frac d2} u_R$, where $u_R$ solves the Helffer-Sj\"ostrand PDE
\begin{equation}
\label{e.HS}
\left\{ 
\begin{aligned}
& 
-\L_{\mu}u_R + \nabla^*\cdot V'' \nabla u_R 
 =  \nabla^* \cdot f_R
\quad &\mbox{in} & \ Q_R\times \Omega_0(Q_R)
\\ & 
u_R = 0& \mbox{on} & \ \partial Q_R \times\Omega_0(Q_R),
\end{aligned}
\right.
\end{equation}
and by testing \eqref{e.HS} with $u_R$ and integration by parts,  we may rewrite the energy functional  $ \mathsf{E}_{\mu_{R},Q_R,f_R} \left[u_R\right]$, and thus \eqref{e.variational} as \cite{NS,AW}
\begin{equation*}
\var_{\mu_R} [R^{-d/2}\sum_{x\in Q_R} \nabla \phi(x) \cdot f_R\left(x \right) ]   
=
\sum_{x\in Q_R} R^{-d} \left\langle  f_R\left(x \right) \nabla u_R  (x)  \right\rangle_{\mu_R}
\end{equation*}

Therefore Theorem \ref{t.qclt} follows from the quantitative homogenization of the Hellfer-Sj\"ostrand equation \eqref{e.HS},  presented below. 

\begin{theorem}
\label{t.HS}
Suppose that $f_R, f$ satisfy the conditions in Theorem \ref{t.qclt}, and let $\a$ be the diagonal matrix with $\a(e,e) = \mathsf V'' (\nabla \phi(e))$, where $\nabla\phi$ is sampled from the  Gibbs measure $\mu_R$ \eqref{e.GL}. Let $u_R, u$ denote respectively the solution to the equations: 
\begin{equation}
%\label{e.HSfullvol}
\left\{ 
\begin{aligned}
& 
-\L_{\mu}u_R + \nabla^*\cdot \a \nabla u_R 
 =  \nabla^* \cdot f_R
\quad &\mbox{in} & \ Q_R\times \Omega_0(Q_R)
\\ & 
u_R = 0& \mbox{on} & \ \partial Q_R \times\Omega_0(Q_R),
\end{aligned}
\right.
\end{equation}
and 
\begin{equation}
\label{e.homog}
\left\{ 
\begin{aligned}
& 
-\nabla \cdot \ahom \nabla u = f \quad &\mbox{in}& \ [0,1]^2
\\ & 
u = 0& \mbox{on} & \ \partial ([0,1]^2).
\end{aligned}
\right.
\end{equation}

Then there exists $\beta = \beta(d,\lambda,\Lambda)>0$, such that 
\begin{equation}
\left\| u_R(R \cdot ) -u(\cdot)\right\|_{ L^2(\frac 1R Q_R, \mu_R)} + \left\| \nabla u_R( R \cdot ) -\nabla u(\cdot)\right\|_{H^{-1}(\frac 1R Q_R, \mu_R)} 
%+ \left\| \a\nabla u_R - \ahom \nabla u\right\|_{H^{-1}(\Zd, \mu)} 
\leq 
CR^{-\alpha}  \|\nabla^*\cdot f_R\|_{L^\infty}  .
\end{equation}

\end{theorem}

Applying Theorem \ref{t.HS} and rescale the domain by $R$, we have
\begin{multline*}
\left|  \var_{\mu_R} [R^{-d/2}\sum_{x\in Q_R} \nabla \phi(x) \cdot f_R\left( x\right) ]   
-\sum_{x\in Q_R} R^{-d} f_R\left(x\right)  \nabla u\left( \frac{x}R \right)  \right|\\
\leq
 \|\nabla^*\cdot f_R \|_{ L^\infty}   \| u_R - u\left( \frac{\cdot}R \right) \|_{\underline L^2(Q_R,\mu_R)}  
\leq 
CR^{-\beta} \|\nabla^*\cdot f_R \|_{ L^\infty}
\end{multline*}

Moreover,  the limit 
\begin{equation*}
\g := \lim_{R\to \infty} \sum_{x\in Q_R} R^{-d} f_R\left(x\right)  \nabla u\left( \frac{x}R \right) 
= \int_{[0,1]^2} f(x) u(x) \, dx
= \int_{[0,1]^2 \times [0,1]^2} f(x) (\nabla^*\cdot \ahom \nabla )^{-1} (x,y) f(y) \, dxdy
\end{equation*}
exists,  by the convergence of Riemann sum to integral,  with a rate of convergence $O(R^{-\alpha})$.  Combining these estimates 
we conclude Theorem \ref{t.qclt}.

%\ww{Add a coupling argument that link infinite volume CLT to finite volume}. 
\subsection{Finite-volume energy quantities}
In this section we recall the energy quantities and their quantitative convergence results established in \cite{AW}.  As can be seen from the variational characterization,  the convergence of the energy quantities will play an essential role in the proof of Theorem \ref{t.qclt}.  
   Define the subadditive energy quantity 
  \begin{equation*}
  \nu_R(Q_R, f, p):= \inf_{w \in H^1_0(Q_R,\mu_{R,p})} \mathsf{E}_{\mu_{R,p},Q_R,f} \left[w\right],
  \end{equation*}
  where $\mu_{R,p}$ denotes the finite volume Gibbs measure in $Q_R$ with an affine boundary condition $\ell_p (x) = p\cdot x$.   In what follows we consider $f=0$ and simply write it as $\nu_R(Q_R,p)$.  As was explained in \eqref{e.variational}, the minimizer of $\nu_R(Q_R,p)$, which we denote as $v(\cdot, Q_R, p)$, solves the Helffer-Sj\"ostrand equation with an affine boundary condition: 
  \begin{equation}
\label{e.BVP.v}
\left\{ 
\begin{aligned}
& \left( -\L_{\mu} + \nabla^* \a \nabla \right) v(\cdot,Q_R,p) 
 =0
& \mbox{in} & \ Q^\circ_R \times \Omega_0(Q_R), 
\\ & 
v(\cdot,Q_R,p) - \ell_p = 0& \mbox{on} & \ \partial Q_R \times\Omega_0(Q_R),
\end{aligned}
\right.
\end{equation}

We recall the fact that~$\nu_R$  are actually quadratic polynomials for all $R\ge 1$, and one may compute the first and second variations of their defining optimization problems. The following lemma is ~\cite[Lemma 5.2]{AW}. 

\begin{lemma}
[Basic properties of $\nu_R$ ]
\label{l.basicprops}
Fix a cube~$Q \subseteq Q_R$.  The quantities $\nu_R(Q,p)$ and its optimizing functions $v(\cdot,Q,p)$ satisfies

\begin{itemize}
\item \emph{Quadratic representation.}
There exist symmetric matrices $\ahom(Q)\in \R^{d\times d}$,  such that 
\begin{equation}
\label{e.quadrep}
\nu_R(Q,p)
= \frac12 p\cdot \ahom(Q) p + \nu_R(Q, 0) \quad \forall p\in\Rd.
\end{equation}
where the matrix $\ahom(Q)$ can be characterized such that for all $p, p' \in\Rd$, 
\begin{equation}
\label{e.ahomQ}
p' \cdot \ahom(Q) p
=
\mathsf{B}_{\mu_{R,p},Q}\left[ \ell_{p'}, v(\cdot,Q,p)\right].
\end{equation}

\item \emph{First variation.} The optimizing functions are characterized as follows: $v(\cdot,Q,p)$ is the unique element of $\ell_p+H^1_0(Q,\mu_{R,p})$ satisfying 
\begin{align}
\label{e.firstvar.nu}
\mathsf{B}_{\mu_{R,p},Q} \left[ v(\cdot,Q,p) ,w \right]
=
0, \quad \forall w \in H^1_0(Q,\mu_{R,p});
\end{align}

\item \emph{Second variation.}
For every~$w\in \ell _{p}+H_{0}^{1}\left( Q,\mu_{R,p}\right)$, 
\begin{equation}
\label{e.quadresp.nu}
\mathsf{E}_{\mu_{R,p}, Q}\left[ w \right] - |Q| \nu_R(Q,p)
= 
\frac12 \mathsf{B}_{\mu_{R,p},Q} \left[ v(\cdot,Q,p) -w, v(\cdot,Q,p) -w\right] 
\end{equation}
\end{itemize}
\end{lemma}
 
  As $R\to \infty$, the subadditive quantity $\nu_R$ is proved to converge with an algebraic rate of convergence.  Define,  for some positive definite matrix $\ahom = \ahom (V)$ and $\chom$,    
  \begin{equation}
\label{e.quadrep.homs}
\overline{\nu}(p) = \frac12p\cdot \ahom p -\chom.
\end{equation}

We have 
  \begin{proposition}[Proposition 6.9 of \cite{AW}]
\label{p.convergence.muL}
There exist~$\beta(d, \lambda, \Lambda)\in \left( 0,\frac12\right]$ and~$C(d, \lambda, \Lambda)<\infty$ such that, for every~$L\in\N$ with $L\leq R$, we have
\begin{equation}
\left| \nu_R (Q_L,p) - \overline{\nu}(p) \right| 
\leq
C  L^{-\beta}. 
\end{equation}
\end{proposition}

Combine with the quadratic representation \eqref{e.quadrep.homs}, this implies 
\begin{corollary}
\label{c.coeff}
There exist $\beta(d,\lambda,\Lambda)\in \left(0,\tfrac12\right]$ and $C(d,\lambda,\Lambda)<\infty$ such that, for every $R\in\N$,
\begin{equation}
\label{e.coeff}
\left|\ahom(Q_R)  - \ahom \right| 
\leq CR^{-\beta} 
\end{equation}
\end{corollary}

\subsection{Estimates on finite-volume correctors}

A direct consequence of Proposition \ref{p.convergence.muL} and the quadratic response \eqref{e.quadresp.nu} implies the quantitative convergence of the solution to the Dirichlet problem \eqref{e.BVP.v} to the affine function.

\begin{lemma}
\label{l.convergencesol}
There exist $\beta(d,\lambda,\Lambda)\in \left(0,\tfrac12\right]$ and $C(d,\lambda,\Lambda)<\infty$ such that, for every $R\in\N$,
\begin{equation}
\frac 1R \left\| v(\cdot, Q_R,  p) - \ell_p \right\|_{\underline{L}^2(Q_R,\mu_R)}
\leq C R^{-\beta} .
\end{equation}
\end{lemma}

An application of the multiscale Poincar\'e inequality (Proposition \ref{p.MP}) implies the quantitative convergence of the fluxes along the geometric scales $R= 3^n$.  We define for every $m\in\N$, 
\begin{equation}
\label{e.triad}
\cu_m:= \left[ -3^m, 3^m \right]^d\cap \Zd.  
\end{equation}
We also define,  for $m\le n$,  $\mathcal{Z}_{m}:= 3^m\Zd \cap \cu_n$, so that $\{ y+ \cu_m: y \in \mathcal Z_m\}$ is a partition of $\cu_n$. 
The next lemma shows the spatial average of the flux is concentrated around its mean. 

\begin{lemma}
\label{l.varest}
There exist $\beta(d,\lambda,\Lambda)\in \left(0,\tfrac12\right]$ and $C(d,\lambda,\Lambda)<\infty$ such that, for every $m\in\N$,
\begin{equation}
\var_{\mu_{\cu_m}} \left[ \left(\a \nabla v(\cdot, \cu_m, p)\right) _{\cu _{m}} \right]
\leq
 C 3^{-m\beta }. 
 \end{equation}
\end{lemma}

\begin{proof}
We define a localized solution $v_{loc}$,  such that if $x\in \cu_m$ is contained in $z+ \cu_{m/3}$ for some $z\in \mathcal Z_{m/3}$,  set $v_{loc}(x) = v(x, z+ \cu_{m/3}, p)$.
We notice that 
\begin{equation*}
\var_{\mu_{\cu_m}} \left[ \left(\a \nabla v(\cdot, \cu_m, p)\right) _{\cu _{m}} \right]
\leq
 \var_{\mu_{\cu_m}} \left[ \left(\a \nabla v_{loc}\right) _{\cu _{m}} \right]+ C 3^{-m\beta/3 }.
 \end{equation*}
 Indeed, it follows from the second variation and triangle inequality that 
 \begin{multline*}
 \left| \var_{\mu_{\cu_m}} \left[ \left(\a \nabla v(\cdot, \cu_m, p)\right) _{\cu _{m}} \right] - \var_{\mu_{\cu_m}} \left[ \left(\a \nabla v_{loc}\right) _{\cu _{m}} \right]\right| \\
 \leq
2 \left\| \a \nabla v(\cdot, \cu_m, p) - \a \nabla v_{loc} \right\|_{\underline L^2(\cu_m, \mu_{\cu_m})}
 \leq
 C(\nu_{3^m}(\cu_{m/3},p) - \nu_{3^m}(\cu_{m},p)),
  \end{multline*}
  and this is bounded by $3^{-m\beta/3}$ by the quantitative convergence of energy (Proposition \ref{p.convergence.muL}).
  
  Therefore it suffices to estimate $\var_{\mu_{\cu_m}} \left[ \left(\a \nabla v_{loc} \right) _{\cu _{m}} \right]$, which we do by using the spectral gap (namely, the Brascamp-Lieb inequality) and that $v_{loc}$ is a localized solution,  to derive a correlation decay for $\a \nabla v_{loc}$.  For simplicity, denote by $\bar {\a\nabla v_{loc}}:=  \left(\a \nabla v_{loc}\right) _{\cu _{m}}$.  Applying the Brascamp-Lieb inequality (Proposition \ref{p.BL}) then yields
  \begin{equation}
  \label{e.BLcum}
 \var_{\mu_{\cu_m}} \left[ \left(\a \nabla v_{loc}\right) _{\cu _{m}} \right]
\leq 
\frac1\lambda 
\sum_{x,y \in \cu_m^\circ}G_{\cu_m}(x,y) \left\langle \left( \partial
_{x}\bar {\a\nabla v_{loc}}\right) \left( \partial _{y}\bar {\a\nabla v_{loc}}\right) \right\rangle_{\mu_{\cu _{m}}}
 \end{equation}
 Notice that by definition \eqref{e.partial}
 \begin{equation*}
  \partial
_{x}\bar {\a\nabla v_{loc}} = \lim_{h \to 0} \frac{\sum_{e\ni x} \a_e (\phi+ h\omega_x) - \a_e(\phi) }{h} \frac{\bar {\a\nabla v_{loc}}(\phi+ h\omega_x) - \bar {\a\nabla v_{loc}}(\phi)}{\sum_{e\ni x} \a_e (\phi+ h\omega_x) - \a_e(\phi) }
\end{equation*}

It follows from the regularity assuption that $\a_e = V''$ is uniformly Lipshitz, and since $\a_e$ is a function of $\nabla \phi$,  we may write $\frac{\sum_{e\ni x} \a_e (\phi+ h\omega_x) - \a_e(\phi) }{h} = \nabla^*\cdot  \mathsf b^h$,  where $\mathsf b^h$ is bounded by a constant independent of $h$.  We claim that there exists $C(d,\lambda,\Lambda)<\infty$, such that 
\begin{equation}
\label{e.barav}
\left\| \frac{\bar {\a\nabla v_{loc}}(\phi+ h\omega_x) - \bar {\a\nabla v_{loc}}(\phi)}{\sum_{e\ni x} \a_e (\phi+ h\omega_x) - \a_e(\phi) } \right\|_{L^2 (\mu_{\cu_m})}
\leq
C3^{-4m/3}.
\end{equation}
This implies, by \eqref{e.BLcum},  and the estimate that $|\nabla_x \nabla_y \cdot G_{\cu_m}(x,y)| \leq  \frac{C}{\max\{1, |x-y|^2\}}$, that
\begin{multline*}
 \var_{\mu_{\cu_m}} \left[ \left(\a \nabla v_{loc}\right) _{\cu _{m}} \right] \\
\leq 
\frac1\lambda 
\sum_{x,y \in \cu_m^\circ}  \left|\mathsf b(x)  \nabla_x^ \cdot G_{\cu_m}(x,y)  \nabla_y^* \cdot  \mathsf b(y) \right|
\left\| \partial
_{x}\bar {\a\nabla v_{loc}} \right\|_{L^2 (\mu_{\cu_m})}
\left\| \partial
_{y}\bar {\a\nabla v_{loc}} \right\|_{L^2 (\mu_{\cu_m})} \\
\leq 
C3^{-8m/3} \sum_{x,y \in \cu_m^\circ}  \left|\mathsf b(x)  \nabla_x^* \cdot G_{\cu_m}(x,y)  \nabla_y  \mathsf b(y) \right|\\
\leq 
C3^{-8m/3} \sum_{x,y \in \cu_m^\circ} \frac{C}{\max\{1, |x-y|^2\}}
\leq 
Cm 3^{-2m/3} 
\end{multline*}
Thus we conclude the lemma.  
To prove \eqref{e.barav}, notice that 
 \begin{multline*}
\left\| \frac{\bar {\a\nabla v_{loc}}(\phi+ h\omega_x) - \bar {\a\nabla v_{loc}}(\phi)}{\sum_{e\ni x} \a_e (\phi+ h\omega_x) - \a_e(\phi) } \right\|_{L^2 (\mu_{\cu_m})}
\leq 
2 \left\| \frac{(\a(\phi+ h\omega_x) (\nabla v_{loc}(\phi+ h\omega_x) -\nabla v_{loc}(\phi) ) )_{\cu_m} }{\sum_{e\ni x} \a_e (\phi+ h\omega_x) - \a_e(\phi) } \right\|_{L^2 (\mu_{\cu_m})} \\
+
2  \left\| \frac{((\a(\phi+ h\omega_x)  -\a(\phi)) \nabla v_{loc}(\phi) ) )_{\cu_m} }{\sum_{e\ni x} \a_e (\phi+ h\omega_x) - \a_e(\phi) } \right\|_{L^2 (\mu_{\cu_m})}
\end{multline*}
And since 
 \begin{equation}
 \label{e.1}
 ((\a(\phi+ h\omega_x)  -\a(\phi)) \nabla v_{loc}(\phi) ) )_{\cu_m} 
 = 3^{-2m} \sum_{e\ni x} (\a_e (\phi+ h\omega_x) - \a_e(\phi)) \nabla v_{loc} (e) 
 \end{equation}
 Therefore 
  \begin{equation*}
   \left\| \frac{((\a(\phi+ h\omega_x)  -\a(\phi)) \nabla v_{loc}(\phi) ) )_{\cu_m} }{\sum_{e\ni x} \a_e (\phi+ h\omega_x) - \a_e(\phi) } \right\|_{L^2 (\mu_{\cu_m})}
   \leq
   C3^{-2m} \sum_{e\ni x} \| \nabla v_{loc} (e) \|_{L^2(\mu_{\cu_m})}
    \end{equation*}
    where the right side is bounded by 
     \begin{equation*}
     C3^{-2m}\sum_{e \in \mathcal E(\cu_{m/3})} \| \nabla v_{loc} (e) \|_{L^2(\mu_{\cu_m})}
     \leq 
     C3^{-2m}3^{2m/3} \nu_{3^m} (\cu_{m/3},p)
     \leq
     C3^{-4m/3}
      \end{equation*}
      
      To estimate the other term, we claim that for the unique $z\in\mathcal Z_{m/3}$ such that $x\in z+\cu_{m/3}$, 
        \begin{equation}
        \label{e.locav}
        \left\| \frac{\nabla v_{loc}(\phi+ h\omega_x) -\nabla v_{loc} }{\sum_{e\ni x} \a_e (\phi+ h\omega_x) - \a_e(\phi) } \right\|_{ L^2 (z+\cu_{m/3}, \mu_{\cu_m})}
        \leq 
        C3^{2m/3}
          \end{equation}
          
        Indeed,  since $\tilde v_{loc}:=  v_{loc}(\phi+ h\omega_x)$ is a solution to the Dirichlet problem:
      \begin{equation}
%\label{e.BVP.v}
\left\{ 
\begin{aligned}
& \left( -\L_{\mu} + \nabla^* \a(\phi+ h\omega_x) \nabla \right)  \tilde v_{loc}
 =0
& \mbox{in} & \ z+\cu_{m/3} \times \Omega_0( z+\cu_{m/3}), 
\\ & 
 \tilde v_{loc} - \ell_p = 0& \mbox{on} & \ \partial  z+\cu_{m/3} \times\Omega_0( z+\cu_{m/3}),
\end{aligned}
\right.
\end{equation}
Testing the equation of $\tilde v_{loc}$ and $ v_{loc}$ and subtract them, we obtain 
\begin{multline*}
\left\langle \sum_{e \in \mathcal E (z+ \cu_{m/3})} \a_e (\phi) \nabla v_{loc}(e) \nabla ( v_{loc}(e)- \tilde v_{loc}(e) ) \right\rangle_{\mu_{\cu_m}} \\
\leq \left\langle \sum_{e \in \mathcal E (z+ \cu_{m/3})} \a_e (\phi+ h\omega_x) \nabla \tilde v_{loc}(e) \nabla ( v_{loc}(e)- \tilde v_{loc}(e) ) \right\rangle_{\mu_{\cu_m}}
 \end{multline*}
 Therefore 
   \begin{equation*}
     \left\|  \nabla \tilde v_{loc} - \nabla v_{loc} \right\| _{L^2 (z+\cu_{m/3}, \mu_{\cu_m})}^2
     \leq 
     \left\langle \sum_{e\ni x} | \a_e (\phi+ h\omega_x) - \a_e (\phi) | \nabla v_{loc}(e) \nabla ( v_{loc}(e)- \tilde v_{loc}(e) ) \right\rangle _{\mu_{\cu_m}}
      \end{equation*}
      thus by Cauchy-Schwarz
       \begin{equation*}
        \left\| \frac{\nabla \tilde v_{loc} -\nabla v_{loc} }{\sum_{e\ni x} \a_e (\phi+ h\omega_x) - \a_e(\phi) } \right\|_{ L^2 (z+\cu_{m/3}, \mu_{\cu_m})}
        \leq
        \sum_{e\ni x} \| \nabla v_{loc} (e) \|_{L^2(\mu_{\cu_m})}
        \leq 
        C3^{2m/3} \nu_{3^m} (\cu_{m/3},p)
       \end{equation*}
       this is \eqref{e.locav}.  Therefore 
       \begin{multline}
       \label{e.2}
       \left\| \frac{(\a(\phi+ h\omega_x) (\nabla v_{loc}(\phi+ h\omega_x) -\nabla v_{loc}(\phi) ) )_{\cu_m} }{\sum_{e\ni x} \a_e (\phi+ h\omega_x) - \a_e(\phi) } \right\|_{L^2 (\mu_{\cu_m})} \\
       \leq 
       C 3^{-2m}   \left\| \frac{\nabla v_{loc}(\phi+ h\omega_x) -\nabla v_{loc} }{\sum_{e\ni x} \a_e (\phi+ h\omega_x) - \a_e(\phi) } \right\|_{ L^2 (z+\cu_{m/3}, \mu_{\cu_m})}
       \leq
       C3^{-4m/3}
        \end{multline}
        Combine the estimates \eqref{e.1} and \eqref{e.2} above we conclude \eqref{e.barav}, and thus finish the proof of the lemma.
\end{proof}

We are ready to prove the convergence of the fluxes with an algebraic rate. 
\begin{lemma}
\label{l.convergenceflux}
There exist $\beta(d,\lambda,\Lambda)\in \left(0,\tfrac12\right]$ and $C(d,\lambda,\Lambda)<\infty$ such that, for every $n\in\N$,
\begin{equation}
3^{-2n}  \left\| \a \nabla v(\cdot, \cu_n,p) - \ahom p \right\|_{\underline{H}^{-1}(\cu_n,\mu_{\cu_n})}
\leq C 3^{-n\beta }. 
\end{equation}
\end{lemma}

\begin{proof}
We first notice, that by the representation of $\ahom (Q)$ in \eqref{e.ahomQ},  and the definition of $\mathsf B$ norm in \eqref{e.defBU},  for all $m\ge 1$ we may write the spatial average of the flux as 
\begin{equation}
\label{e.fluxave}
\langle \left( \a \nabla v(\cdot, \cu_m, p) \right)_{\cu_m} \rangle_{\mu_{\cu_m}}
= \ahom (\cu_m) p.
\end{equation}

 Recall that for $m\le n$,  $\mathcal{Z}_{m}:= 3^m\Zd \cap \cu_n$.   We apply the multiscale Poincar\'e inequality ( Proposition \ref{p.MP}) to obtain

\begin{align}
\label{e.fluxMP}
3^{-2n} \left\| \a \nabla v(\cdot, \cu_n, p) - \ahom p \right\|_{\underline{H}^{-1}(\cu_{n},\mu_{\cu_n})}
 \leq
C3^{-n} \left\|  \a \nabla v(\cdot, \cu_n, p) - \ahom p\right\|_{\underline{L}^{2}(\cu_{n},\mu_{\cu_n})} \\
+C\sum_{m=0}^{n-1}3^{m-n} \left\langle\left( 
\frac1{\left|\mathcal{Z}_{m}\right|}
\sum_{y\in \mathcal{Z}_{m}}
 \left| \left( \a \nabla v(\cdot, \cu_n, p) - \ahom p\right) _{y+\cu _{m}}\right|^{2}  \right)^{1/2}\right\rangle_{\mu_{\cu_n}}. \notag
 \end{align}
The first term on the right side above is bounded by 
 \begin{equation*}
 C3^{-n} \left(\left\|\nabla v(\cdot, \cu_n, p)\right\|_{\underline{L}^{2}(\cu_{n},\mu_{\cu_n})} +|p |^2\right)
 \leq C'3^{-n}
 \end{equation*}
 For the second term,  triangle inequality implies that 
  \begin{multline*}
\frac1{\left|\mathcal{Z}_{m}\right|}
\sum_{y\in \mathcal{Z}_{m}}
 \left| \left(  \a \nabla v(\cdot, \cu_n, p) - \ahom p\right) _{y+\cu _{m}}\right|^{2}  \\
 \leq 
\frac1{\left|\mathcal{Z}_{m}\right|}
\sum_{y\in \mathcal{Z}_{m}}
 \left| \left(\langle  \a \nabla v(\cdot, y+\cu_m, p) \rangle _{\mu_{y+\cu_m}}- \ahom p\right) _{y+\cu _{m}}\right|^{2}\\
 +
\frac1{\left|\mathcal{Z}_{m}\right|}
\sum_{y\in \mathcal{Z}_{m}}
 \left| \left(\langle  \a \nabla v(\cdot, y+\cu_m, p) \rangle _{\mu_{y+\cu_m}}-  \a \nabla v(\cdot, y+\cu_m, p) \right) _{y+\cu _{m}}\right|^{2}\\
 +
\frac1{\left|\mathcal{Z}_{m}\right|}
\sum_{y\in \mathcal{Z}_{m}}
 \left| \left( \a \nabla v(\cdot, \cu_n, p) -\a \nabla v(\cdot, y+\cu_m, p)\right) _{y+\cu _{m}}\right|^{2}  
  \end{multline*}
 By \eqref{e.fluxave} and Proposition \ref{p.convergence.muL}, we have
 \begin{equation}
 \label{e.a}
 \left| \left(\langle  \a \nabla v(\cdot, y+\cu_m, p) \rangle _{\mu_{y+\cu_m}}- \ahom p\right) _{y+\cu _{m}}\right|^{2}
 \leq \left| \ahom(\cu_m) - \ahom \right|^2 |p|^2
 \leq C\left( |p|+|q|+\mathsf{K}_0\right)^2 3^{-2\beta m}
 \end{equation}
 By the second variation,  we have 
 \begin{equation}
 \label{e.b}
 \left\langle \frac1{\left|\mathcal{Z}_{m}\right|}
\sum_{y\in \mathcal{Z}_{m}}
 \left| \left( \a \nabla v(\cdot, \cu_n, p) -\a \nabla v(\cdot, y+\cu_m, p)\right) _{y+\cu _{m}}\right|^{2} \right\rangle_{\mu_{\cu_n}}
 \leq 
 C(\nu_{3^n}(\cu_m,p) - \nu_{3^n}(\cu_n,p))
 \leq
 C3^{-m\beta}
 \end{equation}
 where the last inequality follows from the quantitative convergence of $\nu_{3^n}$ (Proposition \ref{p.convergence.muL}) and triangle inequality. 
 
 We also apply the variance estimate Lemma \ref{l.varest} to conclude that for every $y\in \mathcal Z_m$, 
 \begin{equation}
 \label{e.c}
  \left| \left(\langle  \a \nabla v(\cdot, y+\cu_m, p) \rangle _{\mu_{y+\cu_m}}-  \a \nabla v(\cdot, y+\cu_m, p) \right) _{y+\cu _{m}}\right|^{2}
  \leq
  C3^{-m\beta}
 \end{equation}
 
 Combining \eqref{e.a}, \eqref{e.b} and \eqref{e.c}, we conclude there exists $\beta=\beta(d,\lambda, \Lambda) >0$ and $C=C(d,\lambda, \Lambda) <\infty$ , such that 
 \begin{equation}
 \frac1{\left|\mathcal{Z}_{m}\right|}
\sum_{y\in \mathcal{Z}_{m}}
 \left| \left(  \a \nabla v(\cdot, \cu_n, p) - \ahom p\right) _{y+\cu _{m}}\right|^{2} 
 \leq 
  C3^{-m\beta}
 \end{equation}
 
 Substitute the above estimates into \eqref{e.fluxMP},  and summing over $m$,  we conclude the Lemma. 
\end{proof}

\subsection {Proof of Theorem \ref{t.HS}}

%  In \cite{AW},  the right side of \eqref{e.variational} is denoted as a function of the subadditive quantity 
%  \begin{equation*}
%  \nu_R(Q_R, f, 0):= \inf_{w \in H^1_0(Q_R,\mu_{R})} \mathsf{E}_{\mu_{R},Q_R,f} \left[w\right]
%  \end{equation*}

%  As $R\to \infty$, the subadditive quantity $\nu_R$ is proved to converge with an algebraic rate of convergence (notice that the result of  \cite{AW} is proved under the assumption $\mathsf V\in C^{2,\gamma}$,  and we assume the stronger assumption that $\gamma=1$ here):
%  \begin{proposition}[Proposition 6.9 of \cite{AW}]
%\label{p.convergence.muL}
%There exist~$\beta(\data)\in \left( 0,\frac12\right]$ and~$C(\data)<\infty$ such that, for every~$L\in\N$ with $L\leq R$, we have
%\begin{equation}
%\left| \nu_R (Q_L,f,p) - \overline{\nu}(f,p) \right| 
%\leq
%C \| f\|_{C^{0,1}}^2 L^{-\beta}. 
%\end{equation}
%\end{proposition}
%\ww{adding a remark that $0<\overline{\nu}(f,p)< \infty$,  and thus we denote by $\g$}
%
%Theorem \ref{t.qclt} thus follows from applying Proposition \ref{p.convergence.muL} with $L=R$ and the variational characterization \ref{e.variational}. 
%  

In the previous subsection,  we established the convergence of the solution to a Dirichlet problem with an affine boundary condition,  with an algebraic rate of convergence.  The equation \eqref{e.HS} we would like to homogenize is more general, but if we localize it on a mesoscale, the boundary condition becomes approximately affine. 
 In this section we prove Theorem \ref{t.HS} by estimating the homogenization
error in terms of the error in the convergence of the correctors and fluxes defined in
the previous subsection. The proof goes through a standard, deterministic argument known as the two-scale expansions,  that follows closely along the argument of \cite{AKM, AKMBook}.  

Given $R\ge 1$,  let $m =\inf\{k\in \N: 3^k \ge R\}$.  We may view the solution $u_R$ (and respectively, $u$) to the the equation \eqref{e.HS} (resp. \eqref{e.homog}) as elements in a slightly larger domain $\cu_m \times \Omega_0(\cu_m)$ (resp.  in $U:= \frac 1R \cu_m$).  Specifically,  we set the value of $u_R$ to be $0$ and $v(\cdot, \cu_n,  p)$ (defined in \eqref{e.BVP.v}) to be $\ell_p$ outside $Q_R$,  and set $u =0$ outside $[0,1]^2$.   

We now rescale the problem and study the Dirichlet problem in a fixed domain with mesh size goes to zero.  Denote by $\eps = R^{-1}$ and  for $i=1,\cdots, d$ define the finite volume corrector 
\begin{equation*}
 \chi_{e_i}^{\eps}(\cdot):= v(\cdot, \cu_m,  e_i) - \ell_{e_i}.
  \end{equation*}
For $u$ that solves (\ref{e.homog}), we construct the modified two scale
expansion 
\begin{equation*}
w^{\varepsilon }\left( x\right) =u\left( x\right) +\varepsilon \nabla u\left( x\right) \cdot \chi^\eps(\frac x\eps), \quad \forall x\in U
\end{equation*}%

\textbf{Step 1. } Substitute $w^{\varepsilon }$ into the
Eq. (\ref{e.HS}) to obtain there exist $\beta(d,\lambda,\Lambda)\in \left(0,\tfrac12\right]$ and $C(d,\lambda,\Lambda)<\infty$ such that, 
\begin{equation*}
\left\Vert \nabla^* \cdot \left( \a\nabla w^{\varepsilon }\right) -\mathcal L_\mu w^{\varepsilon }-\nabla \cdot f^{\varepsilon }\right\Vert
_{\underline H^{-1}\left(U, \mu_R\right) }\leq C\|f\|_{L^\infty(U)}\varepsilon ^{\beta }.
\end{equation*}

We begin by computing 
\begin{multline}
\nabla w^{\varepsilon } = \nabla u+\nabla
u\left( \cdot \right)  \cdot \nabla \chi^\eps +
\eps \sum_{j=1}^d \chi^\eps_{e_j} (\frac{\cdot}{\eps}) \nabla \nabla_j u \\
=  \sum_{j=1}^d  \nabla_j u \nabla v(\frac \cdot \eps,  \cu_m, e_j) + 
\eps \sum_{j=1}^d \chi^\eps_{e_j} (\frac{\cdot}{\eps}) \nabla \nabla_j u 
\label{w'}
\end{multline}
where we used the definition of $\chi^\eps_{e_j}$ in the last line. 
This yields
\begin{multline}
\nabla^* \cdot \a \nabla w^\eps
= 
 \sum_{j=1}^d  \nabla_j u \nabla^* \cdot \a  \nabla v(\frac \cdot \eps,  \cu_m, e_j)
 +  \sum_{j=1}^d  \nabla^* \cdot \nabla_j u \left( \a \nabla v(\frac \cdot \eps,  \cu_m, e_j) \right)
 + \eps \nabla^* \cdot \a \left( \sum_{j=1}^d \chi^\eps_{e_j} (\frac{\cdot}{\eps}) \nabla \nabla_j u  \right)
\end{multline}

We also obtain,  using that $u(x) $ has no $\phi$-dependence,  
\begin{equation*}
\mathcal L_\mu w^{\varepsilon }
= 
 \sum_{j=1}^d  \nabla_j u \cdot \mathcal L_\mu  v(\frac \cdot \eps,  \cu_m, e_j)
 \end{equation*}
 
 Therefore 
 \begin{multline}
 \label{e.4term}
 \nabla^* \cdot \a \nabla w^\eps - \mathcal L_\mu w^{\varepsilon }
 = 
 \sum_{j=1}^d  \nabla_j u  \left( \nabla^* \cdot \a  \nabla v(\frac \cdot \eps,  \cu_m, e_j) - \mathcal L_\mu v(\frac \cdot \eps,  \cu_m, e_j)\right)\\
 + 
  \sum_{j=1}^d  \nabla^* \cdot \nabla_j u \left( \a \nabla v(\frac \cdot \eps,  \cu_m, e_j) -\ahom e_j \right) 
  +
   \sum_{j=1}^d  \nabla^* \cdot \nabla_j u \cdot \ahom e_j
   + 
  \eps \nabla^* \cdot \a  \left( \sum_{j=1}^d \chi^\eps_{e_j} (\frac{\cdot}{\eps}) \nabla \nabla_j u  \right)
  \end{multline}
  
  The first term on the right side above vanishes since $v$ solves the equation \eqref{e.BVP.v}.  The third term is 
   \begin{equation*}
    \sum_{j=1}^d  \nabla^* \cdot \nabla_j u \cdot \ahom e_j
    =  \nabla^* \cdot \ahom \nabla u
    = f
    \end{equation*}
    
  To estimate the rest two terms,  notice that $u$ solves the constant coefficient equation \eqref{e.homog},  and therefore  
   \begin{equation*}
   \left\| \nabla^k \nabla_j u \right\|_{L^\infty(U)}
   \leq
    \left\| \int_{[0,1]^2} \nabla^{k+1} \nabla_j G_{\ahom}(x,y) f(y) \, dy \right\|_{L^\infty([0,1]^2)}
    \leq
    C\| f\|_{L^\infty([0,1]^2)},
      \end{equation*}
      where $G_{\ahom}(x,y)$ is the Green's function for the operator $\nabla^* \cdot \ahom \nabla$ in $[0,1]^2 $ with zero boundary condition.

      We apply the quantitative convergence of the corrector (Lemma \ref{l.convergencesol})  and the fluxes (Lemma \ref{l.convergenceflux}) in the rescaled setting,  which implies there there exist $\beta(d,\lambda,\Lambda)\in \left(0,\tfrac12\right]$ and $C(d,\lambda,\Lambda)<\infty$ such that, 
       \begin{multline*}
        \sum_{j=1}^d  \left\| \nabla^* \cdot \nabla_j u \left( \a \nabla v(\cdot,  e_j) -\ahom e_j \right)\right\|_{\underline H^{-1}(U, \mu_R)} \\
        \leq
        C \left\|  \nabla_j u \right\|_{L^\infty} \sum_{j=1}^d  \left\|  \a \nabla v(\cdot,  e_j) -\ahom e_j  \right\|_{\underline H^{-1}(U, \mu_R)}
        \leq
        C\| f\|_{L^\infty([0,1]^2)} \eps^\beta
 \end{multline*}
 and 
   \begin{equation*}
 \left\|  \eps \nabla^* \cdot \a  \left( \sum_{j=1}^d \chi^\eps_{e_j} (\frac{\cdot}{\eps}) \nabla \nabla_j u  \right)\right\|_{L^2(U, \mu_R)}
   \leq
    C\left\| \nabla \nabla_j u \right\|_{L^\infty}  \eps \sum_{j=1}^d \| \chi^\eps_{e_j} \|_{\underline L^2( Q_R, \mu_R)}
      \leq
        C\| f\|_{L^\infty([0,1]^2)} \eps^\beta
 \end{equation*}
 Thus we conclude 
 \begin{equation*}
\left\Vert \nabla^* \cdot \left( \a\nabla w^{\varepsilon }\right) +\mathcal L_\mu w^{\varepsilon }-f\right\Vert
_{\underline H^{-1}(U, \mu_R)}\leq C\varepsilon ^{\beta }.
\end{equation*}
And  Step 1 follows from the fact that $\|\nabla^*\cdot f^\eps(\frac \cdot \eps) - f(\cdot)\|_{L^\infty([0,1]^2)} \le \eps^{\alpha}$. 

\textbf{Step 2} \ \ We deduce that there exist $\beta(d,\lambda,\Lambda)\in \left(0,\tfrac12\right]$ and $C(d,\lambda,\Lambda)<\infty$ such that, 
\begin{equation*}
\left\Vert u^\eps-w^{\varepsilon }\right\Vert
_{\underline H^{-1}(U, \mu_R)}\leq C \|f\|_{L^\infty (U)}\varepsilon ^{\beta }.
\end{equation*}

Since $u$ solves \eqref{e.HS},  we use the coercivity of the $\mathsf B_{\mu,Q}$ norm defined in \eqref{e.defBU} to obtain 
\begin{equation*}
\left\Vert u^\eps-w^{\varepsilon }\right\Vert
_{\underline H^{-1}(U, \mu_R)}^2\leq
C \mathsf{B}_{\mu_R, Q_R}\left[u^\eps-w^\eps,u^\eps - w^\eps\right]
\end{equation*}
Testing \eqref{e.HS} by $u^\eps-w^\eps$ then implies 
\begin{equation*}
\mathsf{B}_{\mu_R, Q_R} \left[ u^\eps-w^\eps ,u^\eps \right]
=
\frac1{|Q_R|}  \sum_{e\in \mathcal{E}(\eps Q_R)} 
\left\langle 
(\nabla u^\eps (\frac e {\eps},\cdot) -\nabla w^\eps(e,\cdot)  f^\eps(\frac e {\eps}))
\right\rangle_{\mu_R},
\end{equation*}

Therefore by the definition of $\mathsf B_{\mu,Q}$ and Poincar\'e inequality, 
\begin{multline*}
\left\Vert \nabla u^\eps- \nabla w^{\varepsilon }\right\Vert
_{L^2\left( U, \mu_R\right) }^2
\leq 
\frac1{|Q_R|}  \sum_{e\in \mathcal{E}(\eps Q_R)} 
\left\langle 
(\nabla u^\eps(\frac e {\eps},\cdot) -\nabla w^\eps(e,\cdot)  f^\eps(\frac e {\eps}))
\right\rangle_{\mu_R}
-  C\mathsf{B}_{\mu_R, Q_R}\left[w^\eps,u ^\eps- w^\eps\right] \\
\leq 
C \left\Vert u^\eps-w^{\varepsilon }\right\Vert
_{H^{1}\left( \eps Q_R, \mu_R \right) }
\left\Vert \nabla^* \cdot \left( \a\nabla w^{\varepsilon }\right) -\mathcal L_\mu w^{\varepsilon }-\nabla^* \cdot f^{\varepsilon }\right\Vert
_{H^{-1}\left(\eps Q_R, \mu_R \right) } \\
\leq
C \left\Vert \nabla u^\eps- \nabla w^{\varepsilon }\right\Vert
_{L^2\left( \eps Q_R, \mu_R \right) } 
\left\Vert \nabla^* \cdot \left( \a\nabla w^{\varepsilon }\right) -\mathcal L_\mu w^{\varepsilon }-\nabla^* \cdot f^{\varepsilon }\right\Vert
_{\underline H^{-1}(U, \mu_R)}
\end{multline*} 
Absorbing $\left\Vert \nabla u^\eps- \nabla w^{\varepsilon }\right\Vert
_{L^2\left( \eps Q_R, \mu_R \right) } $ to the left side and combining with Step 1 we conclude Step 2.

\textbf{Step 3} We show that there exist $\beta(d,\lambda,\Lambda)\in \left(0,\tfrac12\right]$ and $C(d,\lambda,\Lambda)<\infty$ such that, 
\begin{equation*}
\left\Vert w^{\varepsilon }-u\right\Vert _{L^2\left( \eps Q_R, \mu_R \right) } +\left\Vert
\nabla w^{\varepsilon }-\nabla u\right\Vert _{\underline H^{-1}(U, \mu_R)}
%+\left\Vert \a\nabla
%w^{\epsilon }-f^{\varepsilon }-\left( \ahom \nabla u- f\right)
%\right\Vert _{H^{-1}(Q_R,\mu)}
\leq C\varepsilon ^{\beta },
\end{equation*}%

When $f^\eps$ is a constant in $Q_R$,  this just follows from the quantitative convergence of the corrector (Lemma \ref{l.convergencesol}).  Here we simply compute the derivative of the second term in the two-scale expansion.  

Indeed,  we have 
\begin{equation*}
\left\Vert
\nabla w^{\varepsilon }-\nabla u\right\Vert _{\underline H^{-1}(U, \mu_R)}
\leq
C \left\| \eps \sum_{j=1}^d   \chi^\eps_{e_j} (\frac{\cdot}{\eps})  \nabla_j u \right\|_{L^2(U, \mu_R)}
\leq
C\| \nabla_j u \|_{L^\infty([0,1]^2)} \sum_{j=1}^d  \eps \|\chi^\eps_{e_j} (\frac{\cdot}{\eps}) \|_{L^2(\eps Q_R, \mu_R)}
\end{equation*}

Applying Lemma \ref{l.convergencesol} we conclude 
\begin{equation*}
\left\Vert
\nabla w^{\varepsilon }-\nabla u\right\Vert _{\underline H^{-1}(U, \mu_R)}
%+\left\Vert \a\nabla
%w^{\epsilon }-f^{\varepsilon }-\left( \ahom \nabla u- f\right)
%\right\Vert _{H^{-1}(Q_R,\mu)}
\leq C\varepsilon ^{\beta },
\end{equation*}
Since $w^\eps - u \in H_0^1(U, \mu_R)$,  we apply Lemma \ref{l. H-1} to conclude 
\begin{equation*}
\left\Vert w^{\varepsilon }-u\right\Vert _{L^2\left( \eps Q_R, \mu_R \right) } 
\leq 
C
\left\Vert
\nabla w^{\varepsilon }-\nabla u\right\Vert _{\underline H^{-1}(U, \mu_R)}
%+\left\Vert \a\nabla
%w^{\epsilon }-f^{\varepsilon }-\left( \ahom \nabla u- f\right)
%\right\Vert _{H^{-1}(Q_R,\mu)}
\leq C\varepsilon ^{\beta }.
\end{equation*}

\textbf{Step 4} Finally,  we conclude by combining Steps 2 and 3 which yields

\begin{equation*}
\left\Vert \nabla u^{\varepsilon }-\nabla u \right\Vert_{\underline H^{-1}(U, \mu_R)}
\leq
\left\Vert \nabla u^{\varepsilon }-\nabla w^\eps \right\Vert _{\underline H^{-1}(U, \mu_R)}
+ \left\Vert \nabla w^{\varepsilon }-\nabla u \right\Vert _{\underline H^{-1}(U, \mu_R)}
\leq 
C(1+ \|f\|_{L^\infty(U)})\eps^\beta. 
\end{equation*}
Finally,  the bound for $\| u^\eps - u\|_{L^2 (\eps Q_R, \mu_R)}$ follows from the above $\underline H^{-1}$ bound for the gradients and the fact that since $u^\eps - u \in H_0^1(U, \mu_R)$,  we have by Lemma \ref{l. H-1}
\begin{equation*}
\| u^\eps - u\|_{\underline L^2 (\eps Q_R, \mu_R)}
\leq
C \left\Vert \nabla u^{\varepsilon }-\nabla u \right\Vert_{\underline H^{-1}(U, \mu_R)}.
\end{equation*}

\section{Decoupling of the $\nabla \phi$-field}
\label{s.decouple}

\subsection{Approximate harmonic coupling}

For the discrete Gaussian Free Field(GFF), there is a nice
orthogonal decomposition. More precisely, the conditioned field inside the
domain is the discrete harmonic extension of the boundary value to the whole
domain plus an \textit{independent }copy of a \textit{zero boundary}
discrete GFF.

While this exact decomposition does not carry over to general
gradient Gibbs measures, the next result due to Jason Miller, see \cite{Mi},
provides an approximate version.

\begin{theorem}[\protect\cite{Mi}]
\label{t.decouple} Let $D\subset \mathbb{Z}^{2}$ be a simply connected domain
of diameter $R$, and denote $D^{r}=\left\{ x\in D:\text{dist}%
(x,\partial D)>r\right\} $.  Suppose that $f:\partial
D\rightarrow \mathbb{R}$ satisfies $\max_{x\in \partial D}\left\vert f\left(
x\right) \right\vert \leq 2 \left\vert \log R\right\vert ^{2}$.
Let $\phi $ be sampled from the Gibbs measure (\ref{e.GLD}) on $D$
with zero boundary condition, and let $\phi ^{f}$ be sampled from
Gibbs measure on $D$ with boundary condition $f$. Then there exist
constants $c,\gamma ,\delta ^{\prime }\in \left( 0,1\right) $, that only
depend on $V$, so that if $r>cR^{\gamma }$ then the following holds. There
exists a coupling $\left( \phi ,\phi ^{f}\right) $, such that if $\hat{\phi}%
:D^{r}\rightarrow \mathbb{R}$ is discrete harmonic with $\hat{\phi}%
|_{\partial D^{r}}=\phi ^{f}-\phi |_{\partial D^{r}}$, then 
\begin{equation*}
\mathbb{P}\left( \phi ^{f}=\phi +\hat{\phi}\text{ in }D^{r}\right) \geq
1-c R^{-\delta ^{\prime }}.
\end{equation*}
\end{theorem}

Here and in the sequel of the paper, for a set $A\subset \mathbb{Z}^2$ and a
point $x\in \mathbb{Z}^2$, we use $\text{dist}(x,A)$ to denote
the (lattice) distance from $x$ to $A$.  Since
the above theorem  requires that the boundary condition $f$ is not too large, we
introduce the \textquotedblleft good\textquotedblright\ event 
\begin{equation*}
\mathcal{M} =\left\{ \phi :\max_{v\in D}\left\vert \phi
\left( v\right) \right\vert <\left( \log R\right) ^{2}\right\} ,
\end{equation*}%
which is typical since the Brascamp-Lieb inequality implies that $%
\max_{v\in D}|\phi (v)|\leq O\left( \log R\right) $ with high probability.
Indeed,  by applying the exponential Brascamp-Lieb inequality \eqref{e.BL.linexp} and 
a union bound we immediately obtain 

\begin{lemma}
\label{l.bad}There is some $c_{1}>0$, such that $\mathbb{%
P}^{D,0}\left( \mathcal{M} ^{c}\right) \leq \exp \left(
-c_{1}\left( \log R\right) ^{3}\right) $.
\end{lemma}

We will use repeatedly the following consequence of Theorem \ref{t.decouple}%
. It applies to functions $\rho $ such that the integral of $\rho $ against
a harmonic function is always zero.

\begin{lemma}[\cite{BW},  Lemma 2.7]
\label{l.jasonerr}There exists constants $\delta=\delta(V) ,\gamma=\gamma(V) >0$ and $C<\infty$, such that for any
simply connected $D\subset \mathbb{Z}^{2}$ of diameter $R$, any $r>R^{\gamma
}$ and any $\rho :D\rightarrow \mathbb{R}$ supported on $D^{r}$ that
satisfies $\sum_{x\in D^{r}}\rho \left( x\right) f\left( x\right) =0$ for
all functions $f$ harmonic in $D^{r}$, and $\sum_{y\in
D}\left\vert \rho \left( y\right) \right\vert <\infty $, we have for $R$
large enough, 
\begin{eqnarray*}
&&\left\vert \mathbb{E}^{D,f}\left[ \exp \left( \sum_{x\in D}\rho
\left( x\right) \phi ^{f}\left( x\right) \right) 1_{\mathcal{M}}\right] -%
\mathbb{E}^{D,0}\left[ \exp \left(\sum_{x\in D}\rho \left( x\right)
\phi \left( x\right) \right) 1_{\mathcal{M}}\right] \right\vert \\
&\leq &2\exp \left( C\var_{G,D}\left( \sum_{x\in D}\rho
\left( x\right) \phi \left( x\right) \right) \right) R^{-\delta }\text{.}
\end{eqnarray*}%
\end{lemma}

We will apply Lemma \ref{l.jasonerr} to the increment of harmonic averages of the field in $D$,  defined in the section below.  The increments have finite variances and thus changing the boundary to zero only gives an error of order $R^{-\delta}$.

\subsection{Harmonic averages\label{harmonic}}

We will apply Theorem %
\ref{t.decouple} to study the harmonic average of the
$\nabla\phi$ field. Given $B\subset \mathbb{Z}^{2}$, $v\in B$ and $y\in
\partial B$, we denote by $a_{B}\left( v,\cdot \right) $ the harmonic
measure on $\partial B$ seen from $v$. In other words, let $S^{x}$ denote
the simple random walk starting at $x$, and $\tau _{\partial B}=\inf \left\{
t>0:S\left[ t\right] \in \partial B\right\} $, we have 
\begin{equation*}
a_{B}\left( x,y\right) =\mathbb{P}\left( S^{x}\left[ \tau _{\partial B}%
\right] =y\right) .
\end{equation*}

Given $v\in \mathbb{Z}^{2}$ and $R>0$, let $B_{R}\left( v\right) =\left\{
y\in \mathbb{Z}^{2}:\left\vert v_{1}-y_{1}\right\vert ^{2}+\left\vert
v_{2}-y_{2}\right\vert ^{2}<R^{2}\right\} $.  When $v=0$ we simply write $B_{R}\left( v\right)$ as $B_R$. Define the circle
average of the field with radius $R$ at $v$ by 
\begin{equation}
C_{R}\left( v,\phi \right) =\sum_{y\in \partial B_{R}\left( v\right)
}a_{B_{R}\left( v\right) }\left( v,y\right) \phi \left( y\right) .
\label{e.CR}
\end{equation}

We introduce the geometric scales in order to carry out the multiscale 
argument to prove Lemma \ref{l.small} and \ref{l.medium}.  Let $ \gamma=\gamma(V) $ be the constant in Theorem \ref{t.decouple}, define the sequence of numbers $\left\{ r_{k}\right\}
_{k=1}^{\infty }$, $\left\{ r_{k,+}\right\} _{k=0}^{\infty }$ and $\left\{
r_{k,-}\right\} _{k=0}^{\infty }$ by%
\begin{eqnarray}
r_{k} &=&e ^{-k}N ,  \label{scale} \\
r_{k,+} &=&\left( 1+r_k^{-\gamma}\right) r_{k},  \notag \\
r_{k,-} &=&\left( 1-r_k^{-\gamma}\right) r_{k}.  \notag
\end{eqnarray}%
 We also define%
\begin{eqnarray*}
X_{r_{k,+}}\left( v\right) &=&\sum_{r=\left( 1-\frac 14 r_k^{-\gamma}\right)
r_{k,+}}^{\left( 1+\frac 14 r_k^{-\gamma}\right) r_{k,+}}\left(\frac 12 r_k^{1-\gamma}\right)^{-1} C_{r}\left( v,\phi \right) , \\
X_{r_{k,-}}\left( v\right) &=&\sum_{r=\left( 1-\frac 14 r_k^{-\gamma}\right)
r_{k,-}}^{\left( 1+\frac 14 r_k^{-\gamma}\right) r_{k,-}}\left(\frac 12 r_k^{1-\gamma}\right)^{-1} C_{r}\left( v,\phi \right) ,\\
X_{r_{k}}\left( v\right) &=&\sum_{r=\left( 1-\frac 14 r_k^{-\gamma}\right)
r_{k}}^{\left( 1+\frac 14 r_k^{-\gamma}\right) r_{k}}\left(\frac 12 r_k^{1-\gamma}\right)^{-1} C_{r}\left( v,\phi \right) , \notag
\end{eqnarray*}%
where $C_{r}$ is defined in (\ref{e.CR}).  

We also set the increment process $A_k: = X_{r_{k+1}} - X_{r_{k-}}$ 
and the boundary layer error $\mathcal E_k := X_{r_k} - X_{r_{k-}}$.

Notice that the harmonic average process $X$ defined above is slightly different from the one in \cite{BW},  Section 3. Namely, the width of the harmonic average is thinner (of the order $r_k^{1-\gamma}$ instead of $r_k$).  The increment of the harmonic
average process, $A_k$,  is crucial to the multiscale argument below.  We may write 

\begin{equation}
\label{e.ak}
A_k\left( v,\phi \right) =\left(
\sum_{r=\left( 1-\frac 14 r_{k+1}^{-\gamma}\right)
r_{k+1}}^{\left( 1+\frac 14 r_{k+1}^{-\gamma}\right) r_{k+1}}\left(\frac 12 r_{k+1}^{1-\gamma}\right)^{-1}
 -\sum_{r=\left( 1-\frac 14 r_k^{-\gamma}\right)
r_{k,-}}^{\left( 1+\frac 14 r_k^{-\gamma}\right) r_{k,-}}\left(\frac 12 r_k^{1-\gamma}\right)^{-1} \right) \sum_{y\in \partial B_{r}\left( v\right)
}a_{B_{r}\left( v\right) }\left( v,y\right) \phi \left( y\right) .
\end{equation}%
This can be written as $\sum_{y\in B_{r_k}(v)}\rho_{r_k} \left( v,y\right) \phi \left(
y\right) $, where we define%
\begin{equation}
\label{e.rho}
\rho_{r_k} \left( v,y\right) =\left[ \left(\frac 12 r_{k+1}^{1-\gamma}\right)^{-1} \indc_{|r-r_{k+1}| \leq \frac14 r_{k+1}^{1-\gamma}} -\left(\frac 12 r_k^{1-\gamma}\right)^{-1} \indc_{|r-r_{k,-}| \leq \frac14 r_{k}^{1-\gamma}} \right] a_{B_{\left\vert v-y\right\vert
}\left( v\right) }\left( v,y\right) .
\end{equation}%
Here we omit the dependence of $r$ in the definition of $\rho_{r_k}$.

We now state two important properties of the increment process $A_k$.  They are consequences of Theorem \ref{t.decouple},  and the proof are essentially the same as Lemma 3.1 and 3.2 of \cite{BW}.  

\begin{lemma}
\label{lem:harmtest}For any $k\in\N$, and any discrete harmonic function $h$ in $B_{r_k}$, we
have $\sum_{y\in B_{r_k}(v)}\rho_{r} \left( v,y\right) h\left( y\right) =0.$
\end{lemma}
\begin{proof}
Denote $B_{r_k}$ by $D$.  Suppose $h$ takes the boundary value $h|_{\partial D}=H$.
We conclude the proof by showing for any $r \geq 1$%
\begin{equation*}
 \sum_{y\in \partial B_{r}\left(
v\right) }a_{B_{r}\left( v\right) }\left( v,y\right) h\left( y\right)
=h\left( v\right) .
\end{equation*}%
Indeed, since $h$ is harmonic,%
\begin{equation*}
h\left( y\right) =\sum_{z\in \partial D}a_{D}\left( y,z\right)
H\left( z\right) .
\end{equation*}%
Using the fact that 
\begin{equation*}
\sum_{y\in \partial B_{r}\left( v\right) }a_{B_{r}\left( v\right) }\left(
v,y\right) a_{D}\left( y,z\right) =a_{D}\left( v,z\right) ,
\end{equation*}%
we obtain%
\begin{equation*}
\sum_{y\in \partial B_{r}\left(
v\right) }a_{B_{r}\left( v\right) }\left( v,y\right) h\left( y\right)
= \sum_{z\in \partial
D}a_{D}\left( v,z\right) H\left( z\right) =h\left( v\right) .
\end{equation*}
\end{proof}

The following lemma is a consequence of Theorem \ref{t.decouple} and the
lemma above.

\begin{lemma}
\label{l.average}Suppose the same conditions in Theorem \ref{t.decouple} holds.
%Given $v\in Q_N$, $R_{1}>R_{2}>0$, $\varepsilon >0$ such that $\left(
%1+2\varepsilon \right) R_{1}<$dist$\left( v,\partial D_{N}\right) $, $\left(
%1+2\varepsilon \right) R_{2}<\left( 1-2\varepsilon \right) R_{1}$. 
Let $%
\delta $ be the constant from Theorem \ref{t.decouple}. Let $\phi ^{f}$ be
sampled from gradient field (\ref{e.GLD}) on $B_{r_k}$ with boundary condition $f$,  and $\phi ^{0}$ be sampled
from the zero boundary gradient field on $B_{r_k}$. Then, on an event
with probability $1-O\left( r_{k}^{-\delta }\right) $, we have 
\begin{equation*}
A_k\left( v,\phi ^{f}\right)
=A_k\left( v,\phi ^{0}\right) .
\end{equation*}
\end{lemma}

\subsection{Proof of Lemma \ref{l.small} and Lemma \ref{l.medium}}
\label{s.smallmedium}
A key ingredient for the characteristic function estimates (Lemma \ref{l.small} and Lemma \ref{l.medium}) is the decoupling estimates Proposition \ref{ind} below.  

For $r>0$ denote by $\mathbb{P}^{r,0}$ the law of the Ginzburg-Landau field
in $B_{r}\left( v\right) $ with zero boundary condition (and denote by $%
\mathbb{E}^{r,0}$ the corresponding expectation). The basic building block
of Lemma \ref{l.small} and \ref{l.medium} is the following.

\begin{proposition}
\label{ind}There exists $\delta= \delta (V) >0$ and $C=C(d,\lambda, \Lambda) <\infty$, such that for all $|s|^2<\min\{\frac{\delta(\log N -k-1)}{Ck}, \frac 12\}$ and $k \le \log N - \sqrt {\log N}$,%
\begin{equation}
\label{e.ind}
\log \left\langle \exp \left( isX_{r_{k}}\right) \right\rangle_{\mu_N}
=\sum_{j=1}^{k}\log \mathbb{E}^{r_{j},0}\left[ \exp(is A_j)\right]   +\log \left\langle\exp \left( isX_{r_{0}}\right)  \right\rangle_{\mu_N} +O\left( \sum_{j=1}^{k} e^{Cs^2k}r_{j}^{-\delta
}\right) . 
\end{equation}
\end{proposition}

\begin{proof}
We gave an inductive proof,  by running an induction jointly with \eqref{e.ind} and
\begin{equation}
\label{e.lower}
\left\langle \exp \left( isX_{r_{k}}\right) \right\rangle_{\mu_N} \geq
e^{-Cs^2 k}
\end{equation}
for some $C=C(d,\lambda, \Lambda) <\infty$. 
 For the base case $k=0$,  \eqref{e.ind} trivially holds.  To see \eqref{e.lower} holds for $k=0$,  we Taylor expand the exponential and apply the Brascamp-Lieb inequality to obtain 
 \begin{equation*}
 \left\langle \exp \left( isX_{r_{0}}\right) \right\rangle_{\mu_N}
 \geq	
 1- \frac{s^2}{2} \var_{\mu_N} X_{r_{0}}
 \geq 
 1- Cs^2\var_{\mu_{G,Q_N}} X_{r_{0}}
 \geq 
 e^{-C's^2 }
 \end{equation*}
 To keep iterating, we denote by 
$\mathcal M_k := \left\{ \max_{x\in B_{r_{k-1}} } |\phi(x)| \le (\log r_{k-1})^2 \right\}$. By Lemma \ref{l.bad2}, $\P^{N,0} [\mathcal M_k^c] \le N^{-c_1 }$. 
On the event $\mathcal M_k$ , we may write
\begin{multline}
\label{e.cond}
\left\langle \exp \left( isX_{r_{k}}\right)\indc_{\mathcal M_k} \right\rangle_{\mu_N}
=\left\langle\exp \left( isX_{r_{k-1}}\right) \E \left[  \exp \left( isA_{k-1}\right)\exp(is \mathcal E_{k-1})| \mathcal F_{k-1}\right]\indc_{\mathcal M_k}  \right\rangle_{\mu_N}\\
= \left\langle \exp \left( isX_{r_{k-1}}\right) \E \left[  \exp \left( isA_{k-1}\right)| \mathcal F_{k-1}\right]\indc_{\mathcal M_k}  \right\rangle_{\mu_N}\\
+ \left\langle \exp \left( isX_{r_{k-1}}\right) \exp \left( isA_{k-1}\right)\left(\exp(is \mathcal E_{k-1})-1\right)\indc_{\mathcal M_k}  \right\rangle_{\mu_N}
\end{multline}

We conclude using Lemma \ref{l.jasonerr} that 
\begin{equation*}
\E \left[  \exp \left( isA_{k-1}\right)| \mathcal F_{k-1}\right] = \E^{r_{k-1},0} \left[  \exp \left( isA_{k-1}\right)\right] + R_{k-1}
\end{equation*}
such that with probability $1-\P^{N,0} [\mathcal M_k^c]$ ,  there is some $C<\infty$ and $\delta_1>0$ such that $|R_{k-1}| \le 2\exp\left(C_1 \var_{G, B_{r_{k-1}}} A_{k-1}  \right) r_{k-1}^{-\delta_1}\leq Cr_{k-1}^{-\delta_1}$,  and on an event with probability $ \P^{N,0} [\mathcal M_k^c]$, $|R_{k-1}| $ is bounded by $1$.  This implies, in particular, $\left\langle|R_{k-1}|^2\right\rangle_{\mu_N} \leq  Cr_{k-1}^{-\delta_1}$.  For the last term in \eqref{e.cond},  applying the Brascamp-Lieb inequality yields
\begin{equation*}
\left\langle \left|  \exp(is \mathcal E_{k-1})-1 \right|\right\rangle_{\mu_N}\leq
|s| \left\langle \left|  \mathcal E_{k-1}\right| \right\rangle_{\mu_N}
\leq \left\langle \left|  \mathcal E_{k-1}^2\right| ^\frac 12\right\rangle_{\mu_N} \leq
C r_{k-1}^{-\gamma/2}. 
\end{equation*}
The two estimates above implies that for every $k\in \N$,  we have 
\begin{equation}
\label{e.factor}
\left\langle\exp \left( isX_{r_{k}}\right)\indc_{\mathcal M_k} \right\rangle_{\mu_N}
= 
\left\langle \exp \left( isX_{r_{k-1}}\right) \right\rangle_{\mu_N}
 \E^{r_{k-1},0} \left[  \exp \left( isA_{k-1}\right)\right] 
 + O(r_{k-1}^{-\delta_1/2} )+ O( r_{k-1}^{-\gamma/2}). 
 \end{equation}
 Denote $ \min\{\delta_1/2, \gamma/2\}$ as $\delta$.  We now apply the induction hypothesis and the fact that 
 \begin{equation*}
 \E^{r_{k-1},0} \left[  \exp \left( isA_{k-1}\right)\right]  
 \geq
 1- \frac{s^2}{2} \var_{\mu_{B_{r_{k-1}}}} A_{k-1}
 \geq 
 1- Cs^2\var_{\mu_{G,B_{r_{k-1}}}} A_{k-1}
 \geq 
 e^{-C's^2 }
 \end{equation*}
 to obtain  for some $C<\infty$,
 \begin{equation*}
 \left\langle\exp \left( isX_{r_{k}}\right)\indc_{\mathcal M_k} \right\rangle_{\mu_N}
 \geq 
 e^{-Cs^2(k+1)} +O(r_{k-1}^{-\delta})
  \end{equation*}
  therefore \eqref{e.lower} follows from the fact that $e^{-Cs^2(k+1)} > r_{k-1}^{-\delta/2}$, which is implied by the condition $|s|^2<\min\{\frac{\delta(\log N -k-1)}{Ck}, \frac 12\}$, and $\P^{N,0} [\mathcal M_k^c] \le N^{-c_1 }$. Finally,  by taking the logarithm of \eqref{e.factor} we obtain \eqref{e.ind} for the case $k+1$. 
\end{proof}

\begin{lemma}
\label{l.bad2}
There exists $c_1 >0$,  such that for all  $k \le \log N - \sqrt {\log N}$,  $\P^{N,0} [\mathcal M_k^c] \le N^{-c_1 }$. 
\end{lemma}

\begin{proof}
By making a union bound and applying the exponential Brascamp-Lieb \eqref{e.BL.linexp} and the Chebyshev inequality,  we have for all $t\in\R$, 
\begin{equation*}
\P^{N,0} [\mathcal M_k^c] \le
\sum_{x\in B_{r_k}} \P^{N,0} [\phi(x) > (\log r_k)^2]
\leq
r_k^2 e^{-t  (\log r_k)^2}  e^{C t^2 \log N}.
\end{equation*}
Optimize over $t$,  and use the fact that $\log r_k \geq (\log N)^\frac 12$,  we see that there exists $c_1>0$,  such that for $N$ sufficiently large,
\begin{equation*}
\P^{N,0} [\mathcal M_k^c] \le
r_k^2 e^{-\frac{2c_1  (\log r_k)^4}{\log N}}
\leq
 N^{-c_1 }
 \end{equation*}
\end{proof}

We also need an algebraic rate convergence of the variance of $A_k$ stated below. 

\begin{lemma}
\label{l.homogvar}
%Let $\gamma>0$ be sufficiently small. 
There exists $ \g=\g(V) >0$ and $\beta = \beta(d,\lambda, \Lambda) >0$,  such that for all $j= 1, \cdots \log N$, 
\begin{equation*}
\E^{r_{j},0}\left[ A_j^2\right] = \g + O(r_j^{-\beta}).
\end{equation*}
\end{lemma}

\begin{proof}
Recall from \eqref{e.ak} that,  we may write $A_j = \sum_{x \in B_{r_j}} \phi(x) \rho_{r_j}(x)$,  where $ \rho_{r_j}$ is defined in \eqref{e.rho}. 
%\begin{equation*}
%\rho_{r_j}(x) =
%\left(
%\sum_{r=\left( 1-\alpha_j \right) r_{j+1}}^{\left( 1+\alpha_j \right)
%r_{j+1}}f_{\alpha_j }\left( r/r_{j+1}\right) -\sum_{r=\left( 1-\alpha_j
%\right) r_{j,-}}^{\left( 1+\alpha_j \right) r_{j,-}}f_{\alpha_j }\left(
%r/r_{j,-}\right) \right) \sum_{y\in \partial B_{r}\left( x\right)
%}a_{B_{r}\left( x\right) }\left( x,y\right)
%\end{equation*}
Denote by $G_{r_j}(x,\cdot)$ the Dirichlet Green's function in $B_{r_j}$.  We may use the integration by parts $\phi(x) = \sum_{e \in \mathcal E(B_{r_j})} \nabla \phi(e) \nabla G_{r_j}(x,e)$ to write $A_j$ as  
\begin{equation*}
A_j = \sum_{e \in \mathcal E(B_{r_j})} \nabla \phi(e)  \sum_{x\in B_{r_j}} \nabla G_{r_j}(x,e) \rho_{r_j}(x)
\end{equation*}
In order to apply Theorem \ref{t.qclt},  define 
\begin{equation*}
f_{r_j}(e) := r_j \sum_{x\in B_{r_j}} \nabla G_{r_j}(x,e) \rho_{r_j}(x)
\end{equation*}
Thus $\nabla^*\cdot f_{r_j}(x) = r_j \rho_{r_j}(x)$ and $A_j =r_j^{-1} \sum_{e \in \mathcal E(B_{r_j})} \nabla \phi(e)  f_{r_j}(e)$. 
Notice that,  as $r\rightarrow \infty $, the rescaled
harmonic measure 
\begin{equation*}
ra_{B_{r}\left( 0\right) }\left( 0,\cdot \right) \rightarrow 1/2\pi .
\end{equation*}%
Thus as $r\rightarrow \infty $,  $r\rho_r$ converges to 
\begin{equation*}
f(x):= \frac1{2\pi r_{j+1}}\indc_{|x|= r_{j+1}} -  \frac1{2\pi r_{j,-}}\indc_{|x|= r_{j,-}}
 \end{equation*}
 and that $\|r\rho_r -f \|_{L^\infty(r^{-1} B_r)} = O(r^{-1})$.  

Applying Theorem \ref{t.qclt} using that $sup_{j \geq 1}\| r_j \rho_{r_j}\|_{L^\infty} <\infty$,  there exists  $ \g=\g(V) >0$ and $\beta= \beta(d,\lambda, \Lambda)>0$,  such that
\begin{equation*}
\left|\E^{r_{j},0}\left[ A_j^2\right] - \g\right|  \leq Cr_j^{-\beta}.
\end{equation*}
And we conclude the lemma.

\end{proof}

We are now ready to finish the proof of  Lemma \ref{l.small} and \ref{l.medium}. 

\begin{proof}[Proof of Lemma \ref{l.small} ]
Denote by  $s= \frac t{\sqrt{\log N}}$.  We apply Proposition \ref{ind} and stop at $k_2:= \log N - \sqrt{\log N}$.  It follows that 
\begin{multline}
\left\langle \exp(is \phi(0)) \right\rangle_{\mu_N}
= \left\langle \E\left[ \exp(is \phi(0) -isX_{r_{k_2, -}})| \mathcal F_{k_2}\right] \exp(is X_{r_{k_2}})      \exp(is \mathcal E_{k_2}) \right\rangle_{\mu_N}
\end{multline}
We apply Lemma \ref{l.jasonerr} and Lemma \ref{l.bad2} to conclude 
\begin{multline*}
\E\left[ \exp(is \phi(0) -isX_{r_{k_2, -}})| \mathcal F_{k_2}\right]  
= \E\left[ \exp(is \phi(0) -is X_{r_{k_2, -}})\indc_{\mathcal M_{k_2}} | \mathcal F_{k_2}\right]  \\
+ \E\left[ \exp(is \phi(0) -is X_{r_{k_2, -}})\indc_{\mathcal M_{k_2}^c} | \mathcal F_{k_2}\right]  
= \E^{r_{k_2},0} \left[  \exp(is \phi(0) -isX_{r_{k_2, -}}) \right] +R_{k_2},
\end{multline*}
where,  for some $\delta=\delta(V)\in (0,\frac12]$,  $\left\langle|R_{k_2}|^2\right\rangle_{\mu_N} \leq  r_{k_2}^{-\delta} \le e ^{-\delta \sqrt{\log N}}$.  
Thus
\begin{multline}
\label{e.medest}
\left\langle \exp(is \phi(0)) \right\rangle_{\mu_N}  = 
\E^{r_{k_2},0} \left[  \exp(is \phi(0) -isX_{r_{k_2, -}}) \right]  \left\langle \exp(is X_{r_{k_2}})\right\rangle_{\mu_N}
+ \left\langle R_{k_2}  \exp(is X_{r_{k_2}})\right\rangle_{\mu_N}\\
+  \left\langle (\exp(is \mathcal E_{k_2})  -1) \exp(is \phi(0) -isX_{r_{k_2, -}}) \exp(is X_{r_{k_2}}) \right\rangle_{\mu_N}
\end{multline}
The first term on the right side gives the main contribution.  We claim that there exists $\g=\g(V) >0$ and $\beta = \beta(V)>0$,  such that 
\begin{equation}
\label{e.k2}
\left\langle \exp(is X_{r_{k_2}})\right\rangle_{\mu_N}= 
e^{-\frac{t^2}{2\g}} \left(1+ O\left(\frac{t^2}{(\log N)^{\frac 12}}\right) \right)
 \end{equation}
 together with the Brascamp-Lieb inequality,  which implies 
 \begin{multline*}
 1 \ge \E^{r_{k_2},0} \left[  \exp(is \phi(0) -isX_{r_{k_2, -}}) \right] 
 \ge 
 1- \frac{s^2}{2} \var_{
 \mu_{B_{r_{k_2}}}}[\phi(0) -X_{r_{k_2, -}} ]
 \ge 
 1- \frac{C_1s^2}{2} \var_{\mu_{G,B_{r_{k_2}}}}[\phi(0) -X_{r_{k_2, -}} ] \\
 \ge 
 1- C_2 \frac{t^2 }{(\log N)^\frac 12}
 \end{multline*}
 Therefore 
\begin{equation}
\label{e.1st}
  \E^{r_{k_2},0} \left[  \exp(is \phi(0) -isX_{r_{k_2, -}}) \right]  \left\langle \exp(is X_{r_{k_2}})\right\rangle_{\mu_N}
  = e^{-\frac{t^2}{2\g}} \left(1+ O(\frac{t^2}{(\log N)^\frac 12})\right). 
  \end{equation}

To show \eqref{e.k2},  we apply Proposition \ref{ind} with $s= \frac t{\sqrt{\log N}}$ (thus $s^2 \ll \frac{\delta(\log N -k-1)}{Ck} = O((\log N)^{-1/2})$), which yields
\begin{eqnarray}
&&\log  \left\langle\exp \left( isX_{r_{k_2}}\right) \right\rangle_{\mu_N}
=\sum_{j=1}^{k_2}\log \mathbb{E}^{r_{j},0}\left[ \exp(is A_j)\right]   +\log  \left\langle\exp \left( isX_{r_{0}}\right) \right\rangle_{\mu_N}+O\left( \sum_{j=1}^{k_2} e^{Cs^2k}r_{j}^{-\delta
}\right) .  \notag
\end{eqnarray}
Note that $\sum_{j=1}^{k_2} e^{Cs^2k}r_{j}^{-\delta
} \leq Cr_{k_2}^{-\delta} \leq Ce ^{-\delta \sqrt{\log N}}$.  
We also apply the Brascamp-Lieb inequality to see that 
\begin{equation*}
\left| \mathbb{E}^{r_{j},0}\left[ \exp(is A_j)\right]   - 1 -\frac{s^2}{2} \E^{r_{j},0}\left[ A_j^2\right]\right|
\leq
Cs^4 \E^{r_{j},0}\left[ A_j^4\right]
\leq
\frac{Ct^4}{(\log N)^2}  \E^{G, B_{r_{j}}}\left[ A_j^4\right]
= O(\frac{t^4}{(\log N)^2})
\end{equation*}
Thus
\begin{equation*}
\sum_{j=1}^{k_2}\log \mathbb{E}^{r_{j},0}\left[ \exp(is A_j)\right]  
= \sum_{j=1}^{k_2} -\frac{s^2}{2} \E^{r_{j},0}\left[ A_j^2\right] + O(\frac {t^4} {\log N})
\end{equation*}
We apply Lemma \ref{l.homogvar} to conclude that there exists $ \g=\g(V) >0$ and $\beta= \beta(d,\lambda, \Lambda)>0$,  such that
\begin{equation*}
\E^{r_{j},0}\left[ A_j^2\right] = \g + O(r_j^{-\beta}), 
\end{equation*}
this yields
\begin{equation*}
\sum_{j=1}^{k_2} -\frac{s^2}{2} \E^{r_{j},0}\left[ A_j^2\right] 
 = -\frac{t^2}{2\g} \frac{\log N -(\log N)^\frac 12}{\log N} + O((e ^{-\beta \sqrt{\log N}}), 
\end{equation*}
together with the variance estimate that follows from the Brascamp-Lieb inequality, which gives $\log \left\langle \exp \left( isX_{r_{0}}\right) \right\rangle_{\mu_N} = O(\frac 1 {\log N})$,  we conclude \eqref{e.k2}. 

The other terms on the right side of \eqref{e.medest} can be estimated by 
\begin{equation*}
\left| \left\langle R_{k_2}  \exp(is X_{r_{k_2}})\right\rangle_{\mu_N}\right| \leq 
e ^{-\delta \sqrt{\log N}},
\end{equation*}
and
\begin{multline*}
\left\langle  (\exp(is \mathcal E_{k_2})  -1) \exp(is \phi(0) -isX_{r_{k_2, -}}) \exp(is X_{r_{k_2}}) \right\rangle_{\mu_N}
\leq
\frac{s^2}{2}\left\langle
\mathcal E_{k_2}^2\right\rangle_{\mu_N}\\
\leq
\frac{Ct^2}{2 \log N} \left\langle
\mathcal E_{k_2}^2\right\rangle_{\mu_N}
=  O(\frac 1 {\log N})
\end{multline*}
Substitutes the estimates above into \eqref{e.medest} we conclude Lemma \ref{l.small}.

\end{proof}

\begin{proof}[Proof of Lemma \ref{l.medium}]
The proof is very similar to Lemma \ref{l.small} and thus we give a sketch here.  We apply Proposition \ref{ind} and stop at $k_1:= \frac 12 \log N $.  By conditioning, 
\begin{multline}
\left\langle \exp(is \phi(0)) \right\rangle_{\mu_N}
=\left\langle \E\left[ \exp(is \phi(0) -isX_{r_{k_1, -}})| \mathcal F_{k_1}\right] \exp(is X_{r_{k_1}})      \exp(is \mathcal E_{k_1}) \right\rangle_{\mu_N}
\end{multline}
We then apply Lemma \ref{l.jasonerr} and Lemma \ref{l.bad2} to obtain 
\begin{multline}
\label{e.largeest}
\left\langle \exp(is \phi(0)) \right\rangle_{\mu_N}  = 
\E^{r_{k_1},0} \left[  \exp(is \phi(0) -isX_{r_{k_1, -}}) \right]  \left\langle \exp(is X_{r_{k_1}})\right\rangle_{\mu_N}
+ \left\langle R_{k_1}  \exp(is X_{r_{k_1}})\right\rangle_{\mu_N}\\
+  \left\langle  (\exp(is \mathcal E_{k_1})  -1) \exp(is \phi(0) -isX_{r_{k_1, -}}) \exp(is X_{r_{k_1}})\right\rangle_{\mu_N} ,
\end{multline}
where,  for some $\delta=\delta(V)\in (0,\frac12]$,  $\left\langle|R_{k_1}|^2\right\rangle_{\mu_N} \leq  r_{k_1}^{-\delta} \le N^{-\delta/2}$.  
Thus,  
\begin{equation*}
\left|  \left\langle R_{k_1}  \exp(is X_{r_{k_1}})\right\rangle_{\mu_N} \right| \leq 
N^{-\delta/2},
\end{equation*}
and by the Brascamp-Lieb inequality,
\begin{equation*}
\left\langle  (\exp(is \mathcal E_{k_1})  -1) \exp(is \phi(0) -isX_{r_{k_1, -}}) \exp(is X_{r_{k_1}}) \right\rangle_{\mu_N}
\leq
\frac{s^2}{2}\left\langle  
\mathcal E_{k_1}^2\right\rangle_{\mu_N}\\
=  O(N^{-\gamma})
\end{equation*}
Let $\delta>0$ and $C<\infty$ be the constants from Proposition \ref{ind},  and we take $\eps $ sufficiently small so that $\eps^2 \le \frac \delta {2C}$. 
For the first term on the right side of \eqref{e.largeest},  which is absolutely bounded by $\left| \left\langle \exp(is X_{r_{k_1}})\right\rangle_{\mu_N}\right|$,   we apply Proposition \ref{ind} (with $|s|< \eps$) to obtain 
\begin{eqnarray}
&&\log \left\langle \exp \left( isX_{r_{k_1}}\right)\right\rangle_{\mu_N}
=\sum_{j=1}^{k_1}\log \mathbb{E}^{r_{j},0}\left[ \exp(is A_j)\right]   +\log \left\langle \exp \left( isX_{r_{0}}\right)\right\rangle_{\mu_N} +O\left( \sum_{j=1}^{k_1} e^{Cs^2 k_1}r_{j}^{-\delta
}\right) .  \notag
\end{eqnarray}
Since $s^2 \le \frac \delta {2C}$,  we have
\begin{equation*}
 \sum_{j=1}^{k_1} e^{Cs^2 k_1}r_{j}^{-\delta
} \le
r_{k_1}^{-\delta/4}
\le 
N^{-\delta/8}.
\end{equation*}
Notice that, since the distribution of $A_j$ is symmetric for the zero boundary field,
\begin{equation*}
 \mathbb{E}^{r_{j},0}\left[ \exp(is A_j)\right]   \leq
  1 -\frac{s^2}{2} \E^{r_{j},0}\left[ A_j^2\right]
+
\frac{s^4}{24} \E^{r_{j},0}\left[ A_j^4\right]
\end{equation*}
We apply the Brascamp-Lieb inequality to conclude,  there exists an absolute constant $M<\infty$,  such that 
\begin{equation*}
\max_{j=1,\cdots, k_1} \E^{r_{j},0}\left[ A_j^4\right]
\leq
C\max_{j=1,\cdots, k_1} \E^{G, B_{r_{j}}}\left[ A_j^4\right]
\leq M
\end{equation*}
On the other hand,  applying Lemma \ref{l.homogvar} to conclude that there exists $ \g=\g(V) >0$ and $\beta=\beta(d,\lambda, \Lambda)>0$,  such that for each $j=1,\cdots, k_1$,
\begin{equation*}
\E^{r_{j},0}\left[ A_j^2\right] = \g + O(N^{-\beta}), 
\end{equation*}
Thus,  by choosing $\eps>0$ such that $\eps^2  = \min\{ \frac \g {2M}, \frac \delta {2C},  \frac{3\gamma}{2\g}, \frac{3\delta}{4\g} \}$,  we see that for all $|s|<\eps$,  
\begin{equation*}
 \mathbb{E}^{r_{j},0}\left[ \exp(is A_j)\right]   \leq
 1-  \frac{s^2}{3}\g 
 \end{equation*}
 Thus for $N$ sufficiently large, 
 \begin{equation*}
\left\langle \exp \left( isX_{r_{k_1}}\right) \right\rangle_{\mu_N}
 \leq
 (1+O(N^{-\frac \delta 8}) )\prod_{j=1}^{k_1} \left(1-  \frac{s^2}{4}\g \right) +O(N^{-\frac \delta 2}+ N^{-\gamma})
 \leq
 2  e^{-  \frac{s^2}{3}\g \log N}
  \end{equation*}
 
Substitutes these estimates into \eqref{e.largeest} we conclude the Lemma.

\end{proof}

\begin{remark}
Notice that if we aims for a weaker bound \eqref{e.mediumweak}, then instead of applying \ref{l.homogvar} in the proof above,  we only need a uniform lower bound, namely,  $\E^{r_{j},0}\left[ A_j^2\right]   \geq c_1$ for some $c_1>0$.  This can be proved,  for example,  by using the Mermin-Wagner argument as we did in the next section. 
\end{remark}

\section{ A Mermin-Wagner bound}
\label{s.MW}
 
 In this section we prove Lemma \ref{l.MW}. The upper bound \eqref{e. MW} is obtained by using the following Lemma. 
 
 \begin{lemma}
 \label{l.density}
  Let $X$ be a random variable taking values on a unit circle,  and $f:[0, 2\pi) \to [0,1]$ be its density function.  Suppose that $\forall a \in [0,2\pi)$,  we have
  \begin{equation}
  \label{e.den}
  f(a+\frac 12) f (a- \frac 12) \ge c f(a)^2
  \end{equation}
  Then the random variable has a bounded density on the circle,  and $ \int_0^{2\pi} e^{i\theta} f(\theta) \,d\theta \le 1-\eps$. Moreover, if there exist $t>1$ and $C<\infty$,  such that $\forall a,b \in [0,2\pi)$, 
  \begin{equation}
  \label{e.dent}
  f(a+b) f (a- b) \ge e^{-\frac{Cb^2}{t^2}} f(a)^2
  \end{equation}
  then the characteristic function bound can be improved to  $\int_0^{2\pi} e^{i\theta} f(\theta) \,d\theta \le  \frac{C}{t^2}$.
  \end{lemma}
  
  We first prove Lemma \ref{l.MW} based on Lemma \ref{l.density}.  The argument presented below follows closely a Mermin-Wagner type estimate which has been done in e.g., \cite{MW, Pf,  PS}. Denote by $x_k = (2^k,0)$ for $k = 0,\cdots , \log N$. And 
  \begin{equation}
  \mathcal F_k := \sigma(\phi(x) -\phi(y): |x_1|+|x_2|= 2^k \text{ and } |y_1|+|y_2|= 2^k ).
   \end{equation}
   In other words,  $\mathcal F_k$ specifies all the gradients of the field on the boundary of a diamond of radius $2^k$.  A key observation is that conditioned on $\mathcal F_k$, the gradients of the field inside the region $|x_1|+|x_2|< 2^k$ and the gradients in the region  $|x_1|+|x_2|> 2^k$ are independent.   By progressively conditioning the gradients on the layers $\mathcal F_k$, we have 
    \begin{equation}
    \label{e.condMW}
    \left\langle \exp\left( is \phi(0) \right) \right\rangle_{\mu_N} 
    = \left\langle \prod_{k=0}^{\log N} \E\left[\exp\left( is (\phi(x_k) - \phi(x_{k-1}) \right)
    |   \left( \mathcal F_k \right)_{0\le k\le \log N} \right] \right\rangle_{\mu_N} 
     \end{equation}
     
     Thus \eqref{e. MW} follows if we can show there exist $\eps>0$ and $C<\infty$,  such that  $\forall k = 1, \cdots, \log N$, and $\forall t>1$,  
       \begin{equation}
       \label{e.condfk}
       E\left[\exp\left( it (\phi(x_k) - \phi(x_{k-1}) \right)
    |    \left( \mathcal F_k \right)_{0\le k\le \log N}  \right] \leq 
    \min\{1- \eps, \frac{C}{t^2} \}.
     \end{equation}

We show \eqref{e.condfk} using a Mermin-Wagner type argument. Define the deformation $\forall b \in (0,2\pi]$,
  \begin{equation}
  \tau(x) := \left\{
\begin{aligned}
& b, &&  |x_1|+|x_2| \le 2^{k-1}, 
\\ & 
b\left(1- \frac{ |x_1|+|x_2| }{2^k}\right), && 2^{k-1} \le |x_1|+|x_2| \le 2^k,\\
&0, && |x_1|+|x_2| \ge 2^k
\end{aligned} 
\right.  
   \end{equation}
    which interpolates between $b$ and $0$, and does not change the gradients for all  $  \left( \mathcal F_k \right)_{0\le k\le \log N}  $. We also notice that a straightforward calculation yields $\sum_{e \in \mathcal E(Q_N) }(\nabla \tau(e))^2 \le Cb^2$ for some $C<\infty$. 
    
    Let $\phi^+:= \phi+\tau$ and $\phi^-:= \phi-\tau$. We see the densities of $\phi^+$ and $\phi^-$  (conditioning on all the gradients $\left( \mathcal F_k \right)_{0\le k\le \log N} $) satisfies 
    \begin{multline}
    \label{e.denperturb}
    g(\phi^+) g(\phi^-) = \frac{1}{Z_N^2} \exp\left( -\sum_{(x,y) \in \mathcal E(Q_N)} V(\phi(x) - \phi(y) + \tau(x) - \tau(y)) \right) \\ \cdot\exp\left( - \sum_{(x,y) \in \mathcal E(Q_N)} V(\phi(x) - \phi(y) - \tau(x) +\tau(y))\right) \\
    \geq
    \frac{1}{Z_N^2} \exp\left( -2\sum_{(x,y) \in \mathcal E(Q_N)} V(\phi(x) - \phi(y)) - \frac 12 \sup_{x\in\R} V''(x) \sum_{e\in \mathcal E(Q_N)}  (\nabla \tau(e))^2 \right)
    \geq 
    e^{-Cb^2} g(\phi)^2
    \end{multline}
    for some $c>0$.  
    
    For any set $E\subset \Omega_0(Q_N)$ of field configurations,  integrate the above inequality of densities over $E$ and applying the Cauchy-Schwarz inequality,  we obtain 
    \begin{equation}
    \label{e.prob}
    \P\left( \phi^+ \in E |  \left( \mathcal F_k \right)_{0\le k\le \log N}  \right) 
    \cdot \P\left( \phi^- \in E |  \left( \mathcal F_k \right)_{0\le k\le \log N}  \right) 
    \ge
    e^{-Cb^2} \P\left( \phi \in E |  \left( \mathcal F_k \right)_{0\le k\le \log N}  \right)^2.
     \end{equation}
    Set $t=1$ and $b= \frac 12$.  If we denote by $f$ the conditional density of $\exp\left( i (\phi(x_k) - \phi(x_{k-1}) \right)$ on the unit circle (conditioning on all the gradients $\left( \mathcal F_k \right)_{0\le k\le \log N} $) ,  then it follows that \eqref{e.den} is satisfied (for $t=1$). 
    
    To see that \eqref{e.dent} holds for $t>1$,  notice the fact that for $t $ large, we may rescale $\phi \to t\phi$, so that the corresponding potential is given by $V(\frac{\cdot}{t})$, which has second derivative of order $t^{-2}$. Thus following the same argument as above,  one obtains an  improved bound for the conditional densities (for large $t$),
     \begin{equation}
     g(\phi^+) g(\phi^-) \geq e^{-\frac{Cb^2}{t^2}}g(\phi)^2
       \end{equation}
Therefore by applying Lemma \ref{l.density} we conclude that 
  \begin{equation*}
     \E\left[\exp\left( it (\phi(x_k) - \phi(x_{k-1}) \right)
    |    \left( \mathcal F_k \right)_{0\le k\le \log N}   \right] \leq 
    \min\{1- \eps, \frac C{t^2}\}
     \end{equation*}
     And by the conditioning \eqref{e.condMW} we conclude Lemma \ref{l.MW}. 

Finally we give a proof of Lemma \ref{l.density} .

\begin{proof}[Proof of Lemma \ref{l.density} ]
We focus on the proof of $\int_0^{2\pi} e^{i\theta} f(\theta) \,d\theta \le  \frac{C}{t^2}$ as the $1-\eps$ bound is a classical result and its proof can be found in \cite{MW, Pf, PS}.  Using \eqref{e.dent}, we see that the probability density has a ratio bounded by $e^{C/t^2}$ uniformly over the circle .  %Moreover,  there exists $C_1\le \infty$ such that for any $a,x \in [0,2\pi)$,  $ C_1e^{\frac{Cx^2}{t^2}} f(a+x) \ge f(a) \ge e^{-\frac{Cx^2}{t^2}} f(a+x)$, which immediately implies that $f$ is uniformly continuous on the unit circle,  with 

Thus for all $x \in [0,2\pi)$ and any interval $I \subset [0,2\pi)$ of length at most $\pi$, 
\begin{multline*}
\left(  \fint_{I} f(\theta)\, d\theta \right)^2 e^{-\frac{C}{t^2}}
\leq 
\fint_{I} f(\theta+x) f(\theta-x) \, d\theta
\leq
\fint_{I} f(\theta+x) e^{\frac{C}{t^2}}\fint_I f(\phi-x) \, d\phi \, d\theta\\
\leq e^{\frac{C}{t^2}} \left( \fint_{I+x} f(\theta)\, d\theta \right)
  \left( \fint_{I-x} f(\theta)\, d\theta \right),
 \end{multline*}
  thus
\begin{equation*}
\left(  \int_I f(\theta)\, d\theta \right)^2 e^{-\frac{2C}{t^2}}
\leq  \left( \int_{I+x } f(\theta)\, d\theta \right)
  \left( \int_{I-x } f(\theta)\, d\theta \right)
 \end{equation*}
 In particular,  fix $t\in\R$,  by taking $I=[(k-1)\pi/m, k\pi/m)$, for some $m\in \N$, and $k=1, \cdots 2m$  we have 
 \begin{equation*}
 \left(\max_{k=1, \cdots 2m}  \int_{(k-1)\pi/m}^{k\pi/m} f(\theta)\, d\theta \right)^2 
 \leq
 e^{\frac{2C}{t^2}}
 \left(\min_{k=1, \cdots 2m}  \int_{(k-1)\pi/m}^{k\pi/m} f(\theta)\, d\theta \right)
 \left(\max_{k=1, \cdots 2m}  \int_{(k-1)\pi/m}^{k\pi/m} f(\theta)\, d\theta \right)
 \end{equation*}
 Thus
 \begin{equation*}
 \left(\max_{k=1, \cdots 2m}  \int_{(k-1)\pi/m}^{k\pi/m} f(\theta)\, d\theta \right)
 \leq
 e^{\frac{2C}{t^2}}
 \left(\min_{k=1, \cdots 2m}  \int_{(k-1)\pi/m}^{k\pi/m} f(\theta)\, d\theta \right)
 \end{equation*}
 Therefore for each $k$,  we have 
  \begin{multline*}
  \left|  \int_{\frac{(k-1)\pi}m}^{\frac{k\pi}m} e^{i\theta} f(\theta)\, d\theta  + 
   \int_{\pi+ \frac{(k-1)\pi}m}^{\pi+ \frac{k\pi}m} e^{i\theta} f(\theta)\, d\theta \right|
   \leq
    \left| e^{i\frac{k\pi}m} \int_{\frac{(k-1)\pi}m}^{\frac{k\pi}m}  f(\theta)\, d\theta 
    - e^{i\frac{k\pi}m} \int_{\pi+ \frac{(k-1)\pi}m}^{\pi+\frac{k\pi}m}  f(\theta)\, d\theta \right|\\
    + \max_{\theta\in [\frac{(k-1)\pi}m, \frac{k\pi}m]} |e^{i\frac{(k-1)\pi}m}-e^{i\frac{k\pi}m}|
     \left| \int_{(\frac{(k-1)\pi}m}^{\frac{k\pi}m}  f(\theta)\, d\theta  + \int_{\pi+ \frac{(k-1)\pi}m}^{\pi+ \frac{k\pi}m}  f(\theta)\, d\theta \right|
     \\
     \leq 
    \left(1- e^{-\frac{2C}{t^2}} +O(\frac 1m) \right) 
     \left| \int_{\frac{(k-1)\pi}m}^{\frac{k\pi}m}  f(\theta)\, d\theta  + \int_{\pi+ \frac{(k-1)\pi}m}^{\pi+ \frac{k\pi}m}  f(\theta)\, d\theta \right| +O(\frac 1 {m^4t^2})
  \end{multline*}
  
  Summing over $k$ yields
   \begin{equation*}
   \left| \int_0^{2\pi} e^{i\theta} f(\theta) \,d\theta \right|
    \le 
    1- e^{-\frac{2C}{t^2}} +O(\frac 1m)
    \end{equation*}
    by taking $m > \frac 1C t^2$ we conclude the lemma. 
\end{proof}

\appendix 
\section{Multiscale Poincare inequality,}

We state a discrete version of the multiscale Poincare inequality, which provides an estimate of the $H^{-1}\left( \cu_n \right) $ norm of a function in terms of its spatial averages in triadic subcubes.

\begin{proposition}[\cite{AKM, AW}]
\label{p.MP}({Multiscale Poincare inequality}) Let $\mathcal{Z}_{n}=3^{n}\mathbb{Z}%
^{2}\cap \cu _{m}$. Then%
\begin{equation*}
\left\| f\right\|_{H^{-1}( \cu_{m}) }
\leq
C \left\| f\right\|_{\underline{L}^{2}(\cu_{m})}
+C\sum_{n=0}^{m-1}3^{n}\left( 
\frac1{\left|\mathcal{Z}_{n}\right|}
\sum_{y\in \mathcal{Z}_{n}}
 \left| \left( f\right) _{y+\cu _{n}}\right|^{2}  \right)^{1/2}.
\end{equation*}
\end{proposition}

We also record a lemma here that the $\underline H^{-1}$ norm of $\nabla u$ can bound the $L^2$ oscillation of $u$.

\begin{lemma}
\label{l. H-1}
There exists $C(d) <\infty$ such that for every $m\in\N$ and $u\in H^1(\cu_m, \mu_N)$,
\begin{equation*}
\left\| u - (u)_{\cu_m}\right\|_{\underline L^2(\cu_m, \mu_N)} 
\leq
C \left\| \nabla u\right\|_{\underline H^{-1}(\cu_m, \mu_N)} 
\end{equation*}
And, for every $u \in \in H_0^1(\cu_m,\mu_N)$,
\begin{equation*}
\left\| u \right\|_{\underline L^2(\cu_m, \mu_N)} 
\leq
C \left\| \nabla u\right\|_{\underline H^{-1}(\cu_m, \mu_N)} 
\end{equation*}
\end{lemma}

        \subsection*{Acknowledgments}
We thank Scott Armstrong,  Ron Peled,  Tom Spencer and Ofer Zeitouni for helpful discussions, and Ofer Zeitouni for helpful comments on a previous draft of this manuscript. The research was partially supported by the National Key R\&D Program of China and a grant from Shanghai Ministry Education Commission.

        \small
\bibliographystyle{abbrv}
\bibliography{gradphi}

\newcommand{\noop}[1]{} \def\cprime{$'$}
\begin{thebibliography}{10}

\bibitem{AKMu}
S.~Adams, R.~Koteck{\`y}, and S.~M{\"u}ller.
\newblock Strict convexity of the surface tension for non-convex potentials.
\newblock {\em arXiv preprint arXiv:1606.09541}, 2016.

\bibitem{AD2}
S.~Armstrong and P.~Dario.
\newblock Quantitative homogenization of {L}angevin dynamics.
\newblock in preparation.

\bibitem{AKM}
S.~Armstrong, T.~Kuusi, and J.-C. Mourrat.
\newblock The additive structure of elliptic homogenization.
\newblock {\em Invent. Math.}, 208(3):999--1154, 2017.

\bibitem{AKMBook}
S.~Armstrong, T.~Kuusi, and J.-C. Mourrat.
\newblock {\em Quantitative stochastic homogenization and large-scale
  regularity}, volume 352 of {\em Grundlehren der Mathematischen
  Wissenschaften}.
\newblock Springer-Nature, 2019.

\bibitem{AW}
S.~Armstrong and W.~Wu.
\newblock {$C^2$} regularity of the surface tension for the {$\nabla\phi$}
  interface model.
\newblock {\em Comm. Pure Appl. Math.}, 75(2):349--421, 2022.

\bibitem{Bau1}
R.~Bauerschmidt, J.~Park, and P.-F. Rodriguez.
\newblock The discrete {G}aussian model, {I}. renormalisation group flow at
  high temperature.
\newblock {\em arXiv preprint arXiv:2202.02286}, 2022.

\bibitem{Bau2}
R.~Bauerschmidt, J.~Park, and P.-F. Rodriguez.
\newblock The discrete {G}aussian model, {II}. infinite-volume scaling limit at
  high temperature.
\newblock {\em arXiv preprint arXiv:2202.02287}, 2022.

\bibitem{BW}
D.~Belius and W.~Wu.
\newblock Maximum of the {G}inzburg-{L}andau fields.
\newblock {\em Ann. Probab.}, 48(6):2647--2679, 2020.

\bibitem{BS}
M.~Biskup and H.~Spohn.
\newblock Scaling limit for a class of gradient fields with nonconvex
  potentials.
\newblock {\em Ann. Probab.}, 39(1):224--251, 2011.

\bibitem{BL}
H.~J. Brascamp and E.~H. Lieb.
\newblock On extensions of the {B}runn-{M}inkowski and {P}r\'ekopa-{L}eindler
  theorems, including inequalities for log concave functions, and with an
  application to the diffusion equation.
\newblock {\em J. Functional Analysis}, 22(4):366--389, 1976.

\bibitem{BLL}
H.~J. Brascamp, E.~H. Lieb, and J.~L. Lebowitz.
\newblock The statistical mechanics of anharmonic lattices.
\newblock In {\em Statistical Mechanics}, pages 379--390. Springer, 1975.

\bibitem{CS}
J.~G. Conlon and T.~Spencer.
\newblock A strong central limit theorem for a class of random surfaces.
\newblock {\em Communications in Mathematical Physics}, 325(1):1--15, 2014.

\bibitem{DW}
P.~Dario and W.~Wu.
\newblock Massless phases for the {V}illain model in $ d\geq 3$.
\newblock {\em arXiv preprint arXiv:2002.02946}, 2020.

\bibitem{DGI}
J.-D. Deuschel, G.~Giacomin, and D.~Ioffe.
\newblock Large deviations and concentration properties for $\nabla \phi$
  interface models.
\newblock {\em Probability theory and related fields}, 117(1):49--111, 2000.

\bibitem{FS}
T.~Funaki and H.~Spohn.
\newblock Motion by mean curvature from the {G}inzburg-{L}andau {$\nabla \phi$}
  interface model.
\newblock {\em Comm. Math. Phys.}, 185(1):1--36, 1997.

\bibitem{GOS}
G.~Giacomin, S.~Olla, and H.~Spohn.
\newblock Equilibrium fluctuations for {$\nabla\phi$} interface model.
\newblock {\em Ann. Probab.}, 29(3):1138--1172, 2001.

\bibitem{HS}
B.~Helffer and J.~Sj\"ostrand.
\newblock On the correlation for {K}ac-like models in the convex case.
\newblock {\em J. Statist. Phys.}, 74(1-2):349--409, 1994.

\bibitem{MaP}
A.~Magazinov and R.~Peled.
\newblock Concentration inequalities for log-concave distributions with
  applications to random surface fluctuations.
\newblock {\em arXiv preprint arXiv:2006.05393}, 2020.

\bibitem{MW}
N.~D. Mermin and H.~Wagner.
\newblock Absence of ferromagnetism or antiferromagnetism in one-or
  two-dimensional isotropic {H}eisenberg models.
\newblock {\em Physical Review Letters}, 17(22):1133, 1966.

\bibitem{Mi}
J.~Miller.
\newblock Fluctuations for the {G}inzburg-{L}andau {$\nabla\phi$} interface
  model on a bounded domain.
\newblock {\em Comm. Math. Phys.}, 308(3):591--639, 2011.

\bibitem{MiP}
P.~Milo\'{s} and R.~Peled.
\newblock Delocalization of two-dimensional random surfaces with hard-core
  constraints.
\newblock {\em Comm. Math. Phys.}, 340(1):1--46, 2015.

\bibitem{NS}
A.~Naddaf and T.~Spencer.
\newblock On homogenization and scaling limit of some gradient perturbations of
  a massless free field.
\newblock {\em Comm. Math. Phys.}, 183(1):55--84, 1997.

\bibitem{PS}
R.~Peled and Y.~Spinka.
\newblock Lectures on the spin and loop o (n) models.
\newblock In {\em Sojourns in probability theory and statistical physics-{i}},
  pages 246--320. Springer, 2019.

\bibitem{Pf}
C.~E. Pfister.
\newblock On the symmetry of the {G}ibbs states in two-dimensional lattice
  systems.
\newblock {\em Comm. Math. Phys.}, 79(2):181--188, 1981.

\bibitem{She}
S.~Sheffield.
\newblock Random surfaces.
\newblock {\em Ast\'{e}risque}, (304):vi+175, 2005.

\bibitem{WZ}
W.~Wu and O.~Zeitouni.
\newblock Subsequential tightness of the maximum of two dimensional
  {G}inzburg-{L}andau fields.
\newblock {\em Electron. Commun. Probab.}, 24:Paper No. 19, 12, 2019.

\end{thebibliography}

\end{document}